\documentclass[11pt,reqno]{amsart}
\numberwithin{equation}{section}
\usepackage[dvipsnames]{xcolor}
\usepackage{tikz}
\usetikzlibrary{snakes}
\usetikzlibrary{patterns}
\usepackage{amsmath,amsfonts,amssymb,amsbsy,amsopn,amsthm,eucal}
\usepackage{enumerate}
\usepackage{dsfont}
\usepackage{graphicx}   
\usepackage[bookmarks=true, pdfstartview=FitH, colorlinks=true,citecolor=blue, linkcolor=blue]{hyperref}
\usepackage{accents}
\usepackage{enumerate}
\usepackage{enumitem}
\usepackage{color}

\usepackage{cite}

\newcommand{\Vol}{{\rm Vol}}

\newcommand{\Alex}{\text{Alex\,}}

\newcommand{\Alexnk}{\text{Alex}^n(\kappa)}

\newcommand{\dN}{\mathds{N}}

\newcommand{\dR}{\mathds{R}}
\newcommand{\dS}{\mathds{S}}
\newcommand{\dZ}{\mathds{Z}}

\newcommand{\cA}{\mathcal{A}}

\newcommand{\cC}{\mathcal{C}}
\newcommand{\cD}{\mathcal{D}}

\newcommand{\cH}{\mathcal{H}}

\newcommand{\cL}{\mathcal{L}}

\newcommand{\cR}{\mathcal{R}}
\newcommand{\cS}{\mathcal{S}}
\newcommand{\cU}{\mathcal{U}}

\newcommand{\ang}[3]{\measuredangle\left({#1}\,_{#3}^{#2}\right)}

\newcommand{\cang}[4]{\tilde\measuredangle_{#1}\left({#2}\,_{#4}^{#3}\right)}
\newcommand{\geod}[1]{\left[\,#1\,\right]}
\newcommand{\geodii}[1]{\,]\,#1\,[\,}
\newcommand{\geodci}[1]{[\,#1\,[\,}
\newcommand{\geodic}[1]{\,]\,#1\,]}

\newcommand{\dsp}{\displaystyle}

\newtheorem{theorem}{Theorem}[section]

\newtheorem{proposition}[theorem]{Proposition}

\newtheorem{lemma}[theorem]{Lemma}
\newtheorem{corollary}[theorem]{Corollary}
\theoremstyle{definition}
\newtheorem{definition}[theorem]{Definition}
\newtheorem{example}[theorem]{Example}

\theoremstyle{remark}
\newtheorem{remark}{Remark}[section]
\theoremstyle{remark}

\theoremstyle{remark}\newtheorem{conjecture}{Conjecture}[section]
\theoremstyle{remark}
\theoremstyle{remark}

\makeatletter
\def\smalloverbrace#1{\mathop{\vbox{\m@th\ialign{##\crcr\noalign{\kern3\p@}%
  \tiny\downbracefill\crcr\noalign{\kern3\p@\nointerlineskip}%
  $\hfil\displaystyle{#1}\hfil$\crcr}}}\limits}
\makeatother

\begin{document}
\title{Gluing of multiple Alexandrov spaces}

\author[J.~Ge]{Jian Ge*}
\address[Ge]{BICMR, Peking University, Beijing, China}
\email{jge@bicmr.pku.edu.cn}
\thanks{*J. Ge is partially supported by the NSFC grant\#}

\author[N.~Li]{Nan Li**}
\address{N. Li, Department of Mathematics, The City University of New York - NYC College of
Technology, 300 Jay St., Brooklyn, NY 11201}
\email{NLi@citytech.cuny.edu}
\thanks{**Nan Li was partially supported by the PSC-CUNY Research Award \#61533-00 49.}

\date{\today}
\maketitle

\begin{abstract}
  In this paper we discuss the sufficient and necessary conditions for multiple Alexandrov spaces being glued to an Alexandrov space. We propose a Gluing Conjecture, which says that the finite gluing of Alexandrov spaces is an Alexandrov space, if and only if the gluing is by path isometry along the boundaries and the tangent cones are glued to Alexandrov spaces. This generalizes Petrunin's Gluing Theorem. Under the assumptions of the Gluing Conjecture, we classify the $2$-point gluing over $(n-1,\epsilon)$-regular points as local separable gluing and the gluing near un-glued $(n-1,\epsilon)$-regular points as local involutional gluing. We also prove that the Gluing Conjecture is true if the complement of $(n-1,\epsilon)$-regular points is discrete in the glued boundary. In particular, this implies the general Gluing Conjecture as well as a new Gluing Theorem in dimension 2. 
\end{abstract}

\tableofcontents

\section{Introduction}\label{s:intro}

Let $\Alexnk$ denote the collection of $n$-dimensional Alexandrov spaces whose curvature is bounded from below by $\kappa$ in the sense of Toponogov comparisons. By the definition we have that $n$-dimensional Riemannian manifold $M\in\Alexnk$ if the sectional curvature $\sec_M\ge\kappa$. If a group $G$ acts on $X\in\Alexnk$ isometrically with compact orbits, then the quotient space $X/G\in\Alexnk$. Alexandrov spaces can also be constructed by taking suspensions, cones and joins. Note that these constructions are all originated from manifolds. It was asked wether there are non-manifold designated methods to construct or confirm ``non-trivial" Alexandrov spaces. Related work can be found in \cite{LiNa19}, where for every $1\le k\le n-2$, Li and Naber construct Alexandrov spaces that are bi-Lipschitz to $n$-spheres and the $(k,\epsilon)$-singular sets are $k$-dimensional fat Cantor sets. These examples can be viewed as Alexandrov spaces with $C^1$-nontrivial metrics. Constructing new Alexandrov spaces may also be related to the open question wether every Alexandrov space is isometric to a Gromov-Hausdorff limit of Riemannian manifolds with uniform lower curvature bound, if collapsing is allowed. In this paper we study the gluing construction. The following is a classical result along this direction.

\begin{theorem}[\cite{Pet1997}]\label{t:pet.glu}
Let $X_i\in\Alex^n(\kappa)$ with non-empty boundary $\partial X_i$, $i=1,2$. Suppose that there is an isometry $\phi\colon X_1\to X_2$ with respect to their intrinsic metrics. Then the glued space $(X_1\amalg X_2)/ \{x\sim\phi(x)\}\in \Alex^n(\kappa)$.
\end{theorem}

This theorem generalizes Perelman's doubling theorem \cite{Per1991}, which assumes $X_1=X_2$. Note that Petrunin's Theorem requires that every point is glued with exactly one distinct point and the entire boundary has to be glued with another boundary. This type of gluing is called the {\it superable gluing}. By the Globalization Theorem \cite{BGP}, Petrunin's Theorem still holds if the gluing is locally superable. In this paper we consider more general gluing structures, which includes self-gluing and partial gluing of multiple points along multiple Alexandrov spaces, as well as possible mixed types of these gluing in any small domain. The following example illustrates some possible types of gluing that we will deal with.



\begin{figure}[hbt!]
\begin{tikzpicture}
\draw [black, fill=RoyalBlue!80] (-5.1,0)node[left, yshift=0.1cm]{$B_{1}$}  -- (-3.37,1)node[left,yshift=0.2cm]{$O_{1}$} -- (-3.37,3)node[above, xshift=-0.2cm]{$A_{1}$} arc (90:210:2);
\draw [black, fill=RoyalBlue!20] (-1.44,0)node[right]{$A_{2}$} -- (-3.17,1)node[right, yshift=0.2cm]{$O_{2}$} -- (-3.17,3)node[above, xshift=0.2cm]{$B_{2}$}arc(90:-30:2);
\draw [black, fill=RoyalBlue!50] (-5, -0.2)node[left,yshift=-0.2cm]{$A_{3}$}  -- (-3.27, 0.8)node[below,yshift=-0.1cm]{$O_{3}$} -- (-1.54,-0.2cm)node[right,yshift=-0.2cm]{$B_{3}$} arc(-30:-150:2);
\node at (-4,1.5)[left]{$X_{1}$};
\node at (-2.4,1.5)[right]{$X_{2}$};
\node at (-3.5, -.3)[right]{$X_{3}$};
\draw [black, fill=RoyalBlue!80] (0,0) -- (1.73,1) -- (1.73,3) arc (90:210:2);
\draw [black, fill=RoyalBlue!20] (3.46,0) -- (1.73,1)-- (1.73,3)arc(90:-30:2);
\draw [black, fill=RoyalBlue!50] (0,0) -- (1.73,1)-- (3.46,0)arc(-30:-150:2);
\node at (-3, -2) {Three sectors $X_1$, $X_2$ and $X_3$};

\node at (1.73,3)[above]{$C_{1}$};
\node at (0,0)[left]{$C_{3}$};
\node at (3.46,0)[right]{$C_{2}$};
\node at (1.2,1.5)[left]{$X_{1}$};
\node at (2.3,1.5)[right]{$X_{2}$};
\node at (1.4, -.3)[right]{$X_{3}$};
\node at (1.7, 0.5) {$O$};
\draw [dashed] (-0.2, 1.5)node[left]{$P$}--(3.7,0.5)node[right]{$Q$};
\node at (1.7, -2) {The unit disk $Y$};
\draw [black, fill=RoyalBlue, fill opacity=1](4.79, 1.484)node[left, black]{$P$}--(6.73,1)--(6.73, 3)arc(90:166:2);
\draw [black, fill=RoyalBlue!20, fill opacity=1](8.67,0.516)node[right, black]{$Q$}--(6.73,1)--(6.73, 3)arc(90:-14:2);
\draw [black, fill=RoyalBlue!50, fill opacity=1](5.67, 2.7)--(6.73,1)--(8.728,1.07)arc(2:122:2);
\draw [black, fill=RoyalBlue!80, fill opacity=1](4.79, 1.484)--(6.73,1)--(5.67,2.7)arc(122:166:2);
\draw [black, fill=RoyalBlue!20, fill opacity=1](6.73,1)--(8.67,0.516)arc(-14:2:2)--(6.74,1);
\draw [dotted] (6.73,3)node[above]{$C_{1}$}--(6.73, 1);
\filldraw[black](6.73, 1) circle (1pt)node[below]{$O$};
\node at (7, -2) {The doubled half disk $Z$};
\node at (8.728,1.07)[right]{$C_{2}$};
\node at (5.67,2.7)[above,xshift=-2]{$C_{3}$};
\end{tikzpicture}
\caption{Mixed type of gluing}\label{fig:3sec}
\end{figure}
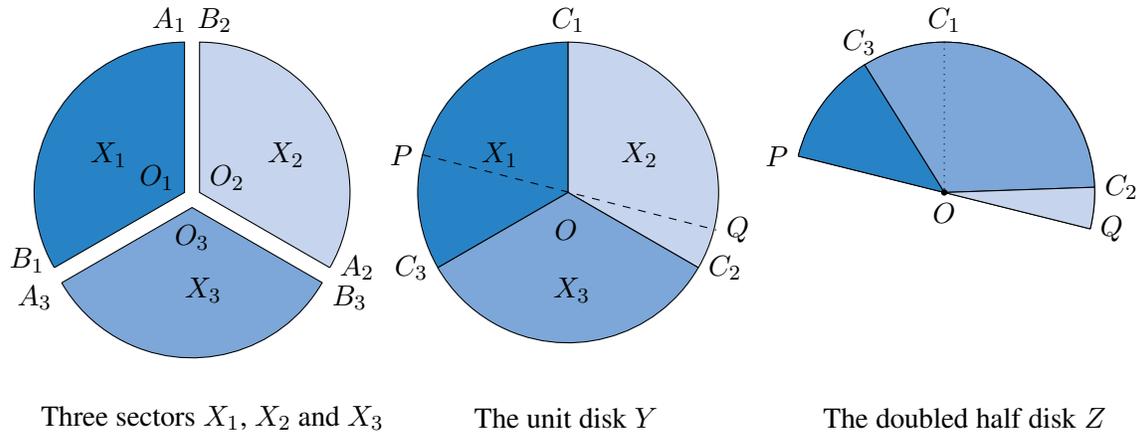


\begin{example}\label{ex:3ptgluing}
Let $X_1=X_2=X_3$ be three identical flat unit sectors with angles $2\pi/3$ at vertices $O_i\in X_i$. Let the unit disk $Y$ be glued from $X_1$, $X_2$ and $X_3$ via identifications $\overline{O_iA_i}\sim \overline{O_{i+1}B_{i+1}}$, where $O_4=O_1$ and $B_4=B_1$. Then glue $\partial Y$ along the reflection about a diameter $\overline{PQ}$. Let $Z\in\Alex^2(0)$ be the glued space, which is simply isometric to the doubled half disk. See Figure \ref{fig:3sec}.

The space $Z$ can be viewed as being glued from $\{X_1, X_2, X_3\}$ via the composition of the above two gluing procedures. In particular, the boundary $\overset\frown{A_1B_1}\subset\partial X_1$ is partially glued with itself and partially glued with $\partial X_3$. Points $O$ and $C_is$ are glued from three points in two or three different spaces. The gluing near $P$ and $Q$ behave like non-superable involutions. Petrunin's Theorem doesn't apply in any small neighborhood of these points. Moreover, Petrunin's Theorem doesn't apply to any pair-wise gluing near $O$ either. This is because any two pairs of $\{X_1, X_2, X_3\}$ glue to a concave space near the vertex.

Without any extra condition, the above mixed types of gluing may show up in any small domain. \hfill $\Box$

\end{example}

If one simply allows all kinds of gluing then there are easy counterexamples such as the gluing of two Alexandrov spaces only at one point from each space. We wish to first discuss some natural and necessary conditions for the glued space being an Alexandrov space. Let us begin with the general notion of gluing.

\subsection{Notion of gluing}

We let $\Alex_\amalg^n(\kappa)$ denote the collection of length metric space $X$, where $X=\amalg_{i=1}^N X_i$ is a disjoint union of finitely many Alexandrov spaces $(X_i, d_{i})\in\Alex^n(\kappa)$. The metric on $X$ is defined as
\begin{align*}
d_X(p,q) &= \renewcommand{\arraystretch}{1.5}
   \left\{\begin{array}{@{}l@{\quad}l@{}}
    d_i(p,q) & \text{ if $p,q\in X_i$ for some $i$; }
    \\
    \infty & \text{ otherwise.}
  \end{array}\right.
\end{align*}
We let $ X^\circ=\amalg_i X_i^\circ$ denote the disjoint union of the interiors of $X_i$ and $\partial X=\amalg_i\partial X_i$ denote the disjoint union of the boundaries of $X_i$. If $\kappa<0$ and $X_i$ is non-compact, we will further require that the volume of spaces of directions on points in $X_i\in\Alexnk$ have a uniform positive lower bound. This is equivalent to \begin{equation}
  \inf_{0<r<1,\, x\in X}\{r^{-n}\Vol(B_r(x))\}\ge c(X)>0,
  \label{intro.e1}
\end{equation}
which always holds if $X$ is compact or $X\in\Alex^n(0)$. This condition is used to prevent the gluing of infinitely many points which locally look more and more like cusps (see Lemma \ref{l:glue.spd} and Lemma \ref{l:G2-1.dense} for the reason of excluding this case). This kind of ``infinite" gluing will not be discussed in this paper.


\begin{definition}[Gluing]\label{d:gluing}
See $\S3$ in \cite{BBI2001} for more details. Let $\mathcal R$ be an equivalence relation on $(X,d)\in \Alex_\amalg^n(\kappa)$. The quotient pseudometric $d_{\mathcal{R}}$ on $X$ is defined as
$$d_{\mathcal{R}}(p,q)
  =\inf\left\{\sum_{i=1}^Nd(p_i,q_i):p_1=p, q_N=q, p_{i+1}\overset {\mathcal{R}}\sim q_i,
  N\in\dN
\right\}.$$
By identifying the points with zero $d_{\mathcal{R}}$-distance, we obtain a length metric space $(Y,d_Y)=(X/d_{\mathcal{R}}, \bar d_{\mathcal{R}})$, which is called the space glued from $\{X_\ell\}$ along the equivalence relation $d_{\mathcal{R}}$. Let $f\colon X\to Y$ be the projection map, which is always 1-Lipschitz onto. Point $x_1$ is said to be glued with $x_2$ if $f(x_1)=f(x_2)$. Note that even if $x_1$ is not equivalent to $x_2$ in the relation $\mathcal R$, they may still be glued together due to the above definition. In this sense, we can always assume that if $x_i\to x$, $y_i\to y$ and $x_i$ and $y_i$ are glued for every $i$, then $x$ is glued with $y$.

Note that every 1-Lipschitz onto map $f\colon X\to Y$ gives rise a gluing structure via equivalence relation $x_1\sim x_2$ if $f(x)=f(x_2)$. Thus we may identify the gluing structure with the projection map $f$, which is also called the gluing map.
\end{definition}

\subsection{Gluing by path isometry}
Let $G_X=\{x\in X\colon |f^{-1}(f(x))|>1\}$ be the set of gluing points. In this paper, we only consider the gluing without losing volume. That is, the Hausdorff measure $\cH^n(G_X)=0$. Now the projection $f\colon X\to Y$ is a 1-Lipschitz onto and volume preserving map. By \cite{NLi15b}, we have the following Lipschitz-Volume Rigidity Theorem.
\begin{theorem}[LV Rigidity, \cite{NLi15b}]\label{t:NLi15b-1-lip}
  For any $X\in\Alex_\amalg^n(\kappa)$ and $Y\in\Alexnk$, if $f\colon X\to Y$ is a 1-Lipschitz onto and volume preserving map, then $f$ preserves the length of paths and $f\,|_{\, X^\circ}$ is an isometry with respect to the intrinsic metrics. Therefore, $Y$ is isometric to a space that is glued from $X$ along $\partial X$.
\end{theorem}
This shows that in order to obtain $Y\in\Alexnk$ by gluing $X\in\Alex_\amalg^n(\kappa)$ without losing volume, the gluing must be along the boundaries ($G_X\subseteq\partial X$) and preserves the lengths of paths. Note that if a curve $\gamma_1$ is glued with a curve $\gamma_2$, then their length $\cL(\gamma_1)=\cL(f(\gamma_1))=\cL(f(\gamma_2))=\cL(\gamma_2)$. If the gluing map $f$ satisfies such a length preserving property, we say that $X$ is \emph{glued by path isometry}. Note that our gluing is defined by a general equivalence relation. A prior, we don't assume that the gluing is continuous in any sense. In particular, we don't assume that if $\gamma\colon(-1,1)\to X$ is a continuous curve with $\gamma\subseteq G_X$, then there is a continuous curve $\sigma\colon(-\epsilon,\epsilon)\to X$ so that $\gamma(t)\neq\sigma(t)$ and $\gamma(t)$ is glued with $\sigma(t)$ for every $-\epsilon< t< \epsilon$. Based on Example \ref{ex:3ptgluing}, the gluing may not be locally separable as required Petrunin's Theorem \ref{t:pet.glu}.

\subsection{Gluing of tangent cones}
Let $\dsp T_{f^{-1}(y)}(X)=\underset{f(x)=y}\amalg T_x(X)=\lim_{r\to0^+}(X,f^{-1}(y),r^{-1}d_X)$ be the disjoint union of tangent cones. For each $r>0$, we have a gluing map $f_{r}\colon (X,f^{-1}(y),r^{-1}d_X)\to (Y,y,r^{-1}d_Y)$, which is induced from $f$, defined on the rescaled spaces. Because $\{(X,x,r^{-1}d_X)\}$ is pre-compact  in pointed Gromov-Hausdorff Topology and $f_{r}$ is distance non-increasing, by Gromov's Compactness Theorem, $\{(Y,y,r^{-1}d_Y)\}$ is also pre-compact. By Arzel\`a-Ascoli Theorem, passing to a subsequence $r_i\to 0^+$, we have a limit gluing map
$$\dsp f_\infty\colon T_{f^{-1}(y)}(X)\to T_y(Y):=\lim_{r_i\to 0^+}(Y, y, r_i^{-1}d_Y),$$
which is also $1$-Lipschitz onto and volume preserving.  In summary, this says that tangent cone $T_y(Y)$ exists when passing to subsequences, and is isometric to a space glued from $T_{f^{-1}(y)}(X)$ along their boundaries and by path isometry. The gluing map $f_\infty\colon T_{f^{-1}(y)}(X)\to T_y(Y)$ is a rescaling limit of the original gluing map $f\colon X\to Y$, depending on the choice of subsequences. Thus the tangent cone $T_y(Y)$ may not be unique. For arbitrary gluing, tangent cone $T_p(Y)$ may not be a metric cone or Alexandrov space. For example, glue two squares $A_1B_1C_1D_1$ and $A_2B_2C_2D_2$ via equivalence relation $a_1\sim a_2$, if $a_1\in\overline{A_1B_1}$, $a_2\in\overline{A_2B_2}$ and $d(a_1,A_1)=d(a_2,A_2)=\frac1i$ for some $i\in\dZ$. However, this won't happen if the glued space $Y\in\Alex^n(\kappa)$, since $T_y(Y)\in\Alex^n(0)$ and is unique.

%

In view of this, it's natural to assume $T_y(Y)\in\Alex^n(\kappa')$ for some $\kappa'$ concerning the generalized Gluing Theorem. Since $T_y(Y)$ is glued from $\dsp T_{f^{-1}(y)}(X)$, such an assumption is equivalent to that $\dsp T_{f^{-1}(y)}(X)$ glues to an Alexandrov space.
The following lemma shows that $T_y(Y)$ is in fact a metric cone if it is known to be a glued Alexandrov space.


\begin{lemma}\label{l:glue.tcone}
  Let $Y$ be a connected length metric space glued from $X\in\Alex_\amalg^n(\kappa)$ by path isometry along the boundary and satisfying (\ref{intro.e1}). If $\dsp T_{f^{-1}(y)}(X)=\lim_{r_i\to0^+}(X,f^{-1}(y), r_i^{-1}d_X)$ glues to a connected Alexandrov space $Z\in\Alex^n(\kappa)$, in terms of the limit gluing map $f_\infty$, then $\dsp Z=T_y(Y)=\lim_{r_i\to0^+}(Y,y, r_i^{-1}d_Y)=C(\Sigma)$ is the metric cone over a connected space $\Sigma\in\Alex^{n-1}(1)$, which is glued from $\Sigma_{f^{-1}(y)}(X)=\underset{f(x)=y}\amalg \Sigma_x(X)$ accordingly.
\end{lemma}

This lemma doesn't follow from the assumptions directly. In our approach it requires Lemma \ref{l:loc.Pet}, a classification of the gluing structure.


\begin{remark}\label{r:glue.tcone.1}
  Lemma \ref{l:glue.tcone} doesn't directly imply that $\Sigma$ is isometric to the space of directions $\Sigma_y$ at $y\in Y$, which should be true if $Y\in\Alexnk$, due to \cite{NLi15b}. In fact, the space of directions $\Sigma_y$ is unknown to be well-defined until $Y$ is proved to be an Alexandrov space. Lemma \ref{l:glue.tcone} doesn't directly imply the uniqueness of $T_y(Y)$ either.
\end{remark}

\begin{remark}\label{r:glue.tcone.2}
  Lemma \ref{l:glue.tcone} rules out the gluing of isolated points. For example, suppose $(Y, y)$ is glued from $(X_1,x_1)$ and $(X_2,x_2)$, where $x_1$ and $x_2$ are the only points to be glued. Then tangent cone $T_{x_1}(X_1)$ is glued with $T_{x_2}(X_2)$ only at the cone points $x_1^*\in T_{x_1}(X_1)$ and $x_2^*\in T_{x_2}(X_2)$. The glued tangent cone $T_y(Y)$ is not an Alexandrov space or a metric cone. The spaces of directions $\Sigma_{x_1}(X_1)$ and $\Sigma_{x_2}(X_2)$ are not glued and the space of directions $\Sigma_{y}(Y)$ is just the disjoint union of $\Sigma_{x_1}(X_1)$ and $\Sigma_{x_2}(X_2)$. In fact, the connectedness of $\Sigma_y$ relies on the tangent cone assumption in Lemma \ref{l:glue.tcone}.
\end{remark}


\subsection{Main results}

In view of the above discussions, it's natural to have the following Gluing Conjecture.

\begin{conjecture}[\cite{NLi15b}]\label{conj:g.Alex}
Let $X\in\Alex_\amalg^n(\kappa)$ satisfying (\ref{intro.e1}), $Y$ be a connected length metric space and $f\colon X\to Y$ be a 1-Lipschitz onto map. Then $Y\in\Alex^n(\kappa)$ if and only if the gluing induced by $f$ is by path isometry along the boundary and the tangent cone $T_{f^{-1}(y)}(X)$ glues to an Alexandrov space for every $y\in Y$ and every limit gluing map $f_\infty$.
\end{conjecture}

\begin{remark}
  The ``if'' part has been proved as in Theorem \ref{t:NLi15b-1-lip}.
\end{remark}

\begin{remark}
  It doesn't make sense to the authors if ``tangent cone glues to an Alexandrov space" is replaced by ``spaces of directions on $X$ glue to an Alexandrov space". This is mainly because the gluing of spaces of directions is not well-defined, in particular when the tangent cones are not known to be glued into metric cones. See Remark \ref{r:glue.tcone.2} for an example.
\end{remark}

\begin{remark}
  Conjecture \ref{conj:g.Alex} can be viewed as a generalization of Petrunin's Theorem \ref{t:pet.glu}. Under the assumptions of Theorem \ref{t:pet.glu}, we have that if $x_1$ glues with $x_2$, then the spaces of directions $\Sigma_{x_1}(X_1)$ glues with $\Sigma_{x_2}(X_2)$ along their boundaries by isometry. This is because the boundary intrinsic geodesics, as quasi-geodesics, are glued with boundary geodesics. By the inductive hypothesis for $(n-1)$-dimensional Alexandrov spaces, we get that they are glued to an Alexandrov space with curv$\ge 1$. Note that the boundary geodesics glue with the boundary geodesics. Thus the tangent cones $T_{x_1}(X_1)$ and $T_{x_2}(X_2)$ are glued along the rays induced by the glued directions, which gives an Alexandrov space with curv $\ge 0$. Then one can apply Conjecture \ref{conj:g.Alex} to conclude Theorem \ref{t:pet.glu} for dimension $n$.
\end{remark}

\begin{remark}\label{r:conj1.3}
This conjecture can't be proved by applying Petrunin's Gluing Theorem pair-wisely. See Example \ref{ex:3ptgluing}.
\end{remark}

In this paper, we prove partial results for Conjecture \ref{conj:g.Alex}, as well as a number of structural results, which remain true in the general case. This is also an important step towards the full solution of Conjecture \ref{conj:g.Alex}. Let $\cR^{n-1}_\epsilon(X)$ be the collection of of $(n-1,\epsilon)$-strained point $p\in\partial X$, whose tangent cone $T_p(X)$ is $\epsilon$-close to splitting off $\dR^{n-1}$ isometrically. Namely, if $p\in\cR^{n-1}_\epsilon(X)$ and $T_p(X)=C(\Sigma_p)$ with cone point $p^*$, then there is a metric space $Z$ and $z\in Z\times\dR^{n-1}$ so that $d_{GH}(B_1(p^*), B_1(z))<\epsilon$. Let $F_X$ be the set of points in $\partial X$ that are not glued with any other point. Let $F_X^\circ$ be the interior of $F_X$ in the boundary topology.

A main difficulty to prove Conjecture \ref{conj:g.Alex} is the possibility of mixed types of gluing in small domain. We overcome this issue by classifying the $2$-point gluing over $\cR^{n-1}_\epsilon(X)$ as local separable gluing and the gluing near $F_X\cap \cR^{n-1}_\epsilon(X)$ as involutional gluing. 
\begin{theorem}[Local gluing structure]\label{t:loc_glue}
  Let $B_R^{\partial X} (p)$ denote the $R$-ball contained in $\partial X$ with respect to the boundary intrinsic metric. Under the assumptions of Conjecture \ref{conj:g.Alex}, there exists $\epsilon_0(X)>0$ small so that the following hold for any $0<\epsilon<\epsilon_0$.
  \begin{enumerate}
  \renewcommand{\labelenumi}{(\roman{enumi})}
  \setlength{\itemsep}{1pt}
  \item For any $p_1, p_2\in G_X\cap\, \cR^{n-1}_\epsilon(X)$ with $f^{-1}(f(p_1))=f^{-1}(f(p_2))=\{p_1, p_2\}$, there exists $R>0$ so that $B_{R}(p_1)\cap B_{R}(p_2)=\varnothing$ and the gluing over $\partial X\cap B_{R}^{\partial X}(p_1)$ and $B_{R}^{\partial X}(p_2)$ is induced by an intrinsic isometry $\phi\colon B_{R}^{\partial X}(p_1)\to B_{R}^{\partial X}(p_2)$.
  \item For any $p\in F_X\cap \cR^{n-1}_\epsilon(X)$, there exists $r>0$ and an isometric involution $\psi\colon B_r^{\partial X}(p)\to B_r^{\partial X}(p)$ with $\psi(p)=p$, so that the gluing over $B_r^{\partial X}(p)$ is induced by the identification $x\sim\psi(x)$.
\end{enumerate}

\end{theorem}

This result provides more detailed structures for the gluing map $f$, which can also be viewed as an extension for the Lipschitz-Volume Rigidity Theorem in \cite{NLi15b}. Its proof relies on the tangent cone conditions in Conjecture \ref{conj:g.Alex}. This phenomenal is in the flavor of controlling local geometry using infinitesimal geometry. Combing these results and the new globalization theorems we develop in this paper, we are able to prove the following theorem.

\begin{theorem}[Main result]\label{t:glue.disc}
Conjecture \ref{conj:g.Alex} holds if $\partial X\setminus(F_X^\circ\cup\cR^{n-1}_\epsilon(X))$ is discrete for any $\epsilon>0$.
\end{theorem}




See Remark \ref{r:glob.des.convex} for the difficulty of $\partial X\setminus\cR^{n-1}_\epsilon(X)$ being not discrete. It's also worth to point out that our result is different from the branched cover problem discussed in \cite{GW2014}, where the gluing structure is explicitly defined, for which the 1-Lipschitz gradient exponential is well-defined in the glued space.

It's an easy exercise (or see \cite{LiNa19} for more general results) to show that $\partial X\setminus\cR^{1}_\epsilon(X)$ is discrete if $\dim(X)=2$. Thus the discreteness condition is automatically satisfied for 2-dimensional Alexandrov spaces. The following result follows from Theorem \ref{t:glue.disc}.

\begin{theorem}
  Conjecture \ref{conj:g.Alex} holds if $n=2$.
\end{theorem}

Moreover, in dimension 2, the tangent cone condition can be reduced to the restrictions on the sum of the glued angles, which are much easier to check. With some extra work, Theorem \ref{t:glue.disc} implies the following result. Suppose that $Y$ is glued from $X\in\Alex_\amalg^2(\kappa)$. Let $\partial Y$ be the collection of $y\in f(\partial X)$ so that some part of the boundary near $f^{-1}(y)$ is preserved over the gluing. That is, there is a curve $\gamma\colon[0,\epsilon)\to \partial X$ so that $f(\gamma(0))= y$ and $f^{-1}(f(\gamma(t)))=\gamma(t)$ for every $0<t<\epsilon$. Let $\Theta_x$, $x\in X$, denote the cone angle of the tangent cone $T_x(X)$, which is equal to the parameter of the space of directions $\Sigma_{x}(X)$.

\begin{theorem}[Gluing of 2-dimensional Alexandrov spaces]\label{t:glue.dim2}
Let $Y$ be a connected length metric space glued from $X\in\Alex_\amalg^2(\kappa)$ satisfying (\ref{intro.e1}) by path isometry along the boundary, and without isolated gluing point. If $\dsp\sum_{f(x)=y}\Theta_{x}\le2\pi$ for every $y\in Y$ and $\dsp\sum_{f(x)=y}\Theta_{x}\le\pi$ for every $y\in\partial Y$, then $Y\in\Alex^2(\kappa)$.
\end{theorem}

This theorem is a full generalization for Petrunin's Theorem \ref{t:pet.glu} in dimension 2, as it allows mixed types of separable gluing, partial gluing and self-gluing of multiple points along multiple Alexandrov spaces, as long as the angles are matched up correctly. In particular, Example \ref{ex:3ptgluing} can be handled by this result.

\begin{remark}
  See Remark \ref{r:glue.tcone.2} for the reason to exclude the gluing of isolated points. 
\end{remark}

\begin{remark}
  The gluing problem of polygons in Euclidean spaces is a classical problem. For example in Chapter VI of \cite{AZ1967}, the authors studied the gluing of polygons. It was known that polygons in $\dR^2$ glue to Alexandrov spaces by path isometry along the boundary if and only if the glued angles satisfy the conditions in Theorem \ref{t:glue.dim2}. In this sense, Theorem \ref{t:glue.dim2} is a generalization of such a result, which allows the gluing of any 2 dimensional Alexandrov spaces.
\end{remark}



In the course of the proof for the main results, we develop a number of new theorems on the gluing and globalization of Alexandrov spaces. Let us discuss some of them below.

An open set $U\subseteq Y$ is said to be an Alexandrov $\kappa$-domain, if $\kappa$-Toponogov comparison holds for any geodesic triangle in $U$. Let $\Omega\subseteq Y$ be a set. By the notation $\Omega\in\Alex^n_{loc}(\kappa)$ we mean that for any $p\in \Omega$ there is an $n$-dimensional Alexandrov $\kappa$-domain $U$ so that $p\in U\subseteq Y$. The following theorem is a generalization of the results on an involutional gluing theorem of cones (see Theorem B in \cite{LR2012}).

\begin{theorem}[Involutional Gluing] \label{t:Self.Glue}
Let $B_{R}(x)\subseteq X$ be an Alexandrov $\kappa$-domain with an isometric involution $\phi\colon B_{R}(x)\cap\partial X\to B_{R}(x)\cap\partial X$. Let $B_{R}(x)/(x\sim\phi(x))$ be the quotient space and $f\colon B_{R}(x)\to B_{R}(x)/(x\sim\phi(x))$ be the projection map. Then $B_{R/10}(f(x))$ is an Alexandrov $\kappa$-domain.
\end{theorem}

Note that $\Omega\in\Alex^n_{loc}(\kappa)$ doesn't imply that its metric completion $\bar\Omega$ is an Alexandrov space, even if $\Omega$ is full measure in $\bar\Omega$. For example, any open domain $\Omega\subset\dR^n$ satisfies $\Omega\in\Alex^n_{loc}(0)$, but the closure $\bar{\Omega}$ is not Alexandrov if $\Omega$ is strictly concave. See \cite{BGP}, \cite{Pet2016} and \cite{NLi15a} for more work on the globalization theorems concerning this issue. In this paper, we prove the following globalization theorem.

\begin{theorem}[Globalization]\label{t:glob.des}
Let $Y$ be a length metric space with $\cU\subseteq Y$ and $\cA=Y\setminus \cU$. Then $Y\in\Alex^n(\kappa)$ if the following are satisfied.
\begin{enumerate}
\renewcommand{\labelenumi}{(\arabic{enumi})}
\setlength{\itemsep}{1pt}
\item $\cU\in\Alex_{loc}^n(\kappa)$.
\item $\cA$ is discrete and for every point $p\in\cA$, the tangent cone $T_p(Y)$ exists and isometric to a metric cone $C(\Sigma)\in\Alex^{n}(0)$.
\end{enumerate}
\end{theorem}

In general, we have the following conjecture, which implies Conjecture \ref{conj:g.Alex} combining with our results established in Section \ref{s:sep.glue} and Section \ref{s:inv.glue}.
\begin{conjecture}\label{c.glob}
Let $Y$ be an $n$-dimensional length metric space. If the following are satisfied for $\cU\subseteq Y$ and $\cA=Y\setminus\cU$, then $Y\in\Alex^n(\kappa)$.
\begin{enumerate}
\renewcommand{\labelenumi}{(\arabic{enumi})}
\setlength{\itemsep}{1pt}
\item $\cU\in\Alex_{loc}^n(\kappa)$.
\item $\cH^{n-1}(\cA)=0$ and for every point $p\in\cA$, the tangent cone $T_p(Y)$ exists and isometric to a metric cone $C(\Sigma)\in\Alex^{n}(0)$.
\end{enumerate}
\end{conjecture}

\subsection{Convention and notation}
Let $X$ be a length metric space.
\begin{enumerate}
  \renewcommand{\labelenumi}{(\arabic{enumi})}
  \setlength{\itemsep}{1pt}
  \item Let $C(\Sigma)$ denote the metric cone over metric space $\Sigma$.
  \item $\Sigma_{p}(X)$ denoted the space of directions at point $p\in X$. We may omit the $X$ and write it as $\Sigma_p$ is no confusion arises.
  \item $T_{p}(X)=C(\Sigma_p)$ denoted the tangent cone at point $p\in X$, whose cone point in $T_p(X)$ is denoted by $p^*$. We may omit the $X$ and write it as $T_p$ is no confusion arises.
  \item Let $\dim_\cH(A)$ denote the Hausdorff dimension of $A\subseteq X$.
  \item Let $\cH^k(A)$ denote the $k$-dimensional Hausdorff measure of $A\subseteq X$.
  \item We let $\geod{ab}$ or sometimes $\geod{a, b}$ denote a geodesic connecting points $a$ and $b$..
\end{enumerate}



\section{Outline of the Proof}\label{s:outline}

From now on the gluing will always satisfy the assumptions in Conjecture \ref{conj:g.Alex}. That is, the spaces are glued by path isometry along the boundary, and every limit gluing $f_\infty$ on tangent cones produces Alexandrov spaces.

\begin{definition}\label{def:angle}
Let $Y$ be a length metric space and $\gamma\colon[0,1]\to Y$ and $\sigma\colon[0,1]\to Y$ be unit speed geodesics with $\gamma(0)=\sigma(0)=p$. Define the angle between geodesics $\gamma$ and $\sigma$ by $\dsp \ang{p}{\gamma}{\sigma}=\limsup_{t_i, s_i\to 0^+}\cang{\kappa}{p}{\gamma(t_i)}{\sigma(s_i)}$, which is the upper limit of the comparison angles in the space form $\dS_\kappa^2$. It's easy to see by the Taylor expansion of law of cosine that the limit doesn't depend on the choice of $\kappa$.
\end{definition}

\begin{definition}
  A point $p\in Y$ is said to be convex if for any geodesic $\geodii{x,y}\ni p$ and any $z\in Y$ with $z\neq p$, the angles satisfy
  \begin{align}
    \ang{p}{\geod{pz}}{\geod{px}}+\ang{p}{\geod{pz}}{\geod{py}}\le\pi.
  \end{align}
\end{definition}

\begin{remark}\label{r:alex.convex}
  By the definition, if for every sequence of rescaling, tangent cone $T_y(Y)$ at $y\in Y$ exists passing to a subsequence and is isometric to a metric cone $C(\Sigma)$ with $\Sigma\in\Alex^{n-1}(1)$, then $p$ is convex. Such a condition holds if $Y\in\Alexnk$ or Lemma \ref{l:glue.tcone} holds. 
\end{remark}


To prove $Y\in\Alex^n(\kappa)$, we will follow the following framework, which was described in \cite{NLi15a}.
\begin{enumerate}[leftmargin=1in,itemsep=0pt,labelsep=12pt]
  \renewcommand{\labelenumi}{(Step \arabic{enumi})}
  \setlength{\itemsep}{1pt}
  \item Decompose $Y$ as $Y=\cU\cup\cA$ with the following two properties.
    \begin{enumerate}
      \renewcommand{\labelenumi}{(\arabic{enumi})}
      \setlength{\itemsep}{1pt}
      \item $\cU\in\Alex_{loc}^n(\kappa)$.
      \item Every point in $\cA$ is a convex point.
\end{enumerate}
  \item Prove that $Y\in\Alexnk$ by some globalization theorems such as Conjecture \ref{c.glob}.
\end{enumerate}

Now let us explain the constructions of $\cU$ and $\cA$. Let $G_Y=\{y\in Y\colon |f^{-1}(y)|>1\}$ and $G_X=f^{-1}(G_Y)$ be the collection of gluing points. For $m\ge 2$, let
$$
G_Y^m=\left\{y\in f(\partial X)\colon\; \left|f^{-1}(y)\right|=m\right\} \text{\quad and \quad} G_X^m=f^{-1}(G_Y^m).
$$
It follows from Lemma \ref{l:G2-1.dense} that $G_Y=\cup_{m=2}^{N_0}G_Y^m$ for some $N_0=N_0(X)<\infty$. Let $F_Y=\{y\in f(\partial X)\colon |f^{-1}(y)|=1\}$. Then $F_X=f^{-1}(F_Y)$ is the set of points on $\partial X$ that are not glued with any other point. Note that these sets may not be closed or open.

\begin{example}\label{e:non-open.close}
In Example \ref{ex:3ptgluing}, if we glue $\{X_1, X_2, X_3\}$ to $Y$, then $$G^{2}_{X}=\overline{O_1A_1}\cup\overline{O_2A_2}\cup\overline{O_3A_3}\setminus\{O_1, O_2, O_3\}$$ is neither open nor closed. The non-glued points $F_X$ is open.

In the gluing of $Y$ to $Z$, we have $F_Y=\{P,Q\}$ is closed. \hfill $\Box$
\end{example}

The following are the main components in $\cU$.
\begin{enumerate}
  \renewcommand{\labelenumi}{(\arabic{enumi})}
  \setlength{\itemsep}{1pt}
  \item $f(F_X^\circ)$, where $F_X^\circ$ is the interior of $F_X$ in the boundary topology.
  \item $G_Y^2(\epsilon)=\{y\in G_Y^2:f^{-1}(y)\subseteq \cR^{n-1}_\epsilon(X)\}$. Let $G_X^2(\epsilon)=f^{-1}(G_Y^2(\epsilon))$. Note that $G_X^2(\epsilon)\subseteq G_X^2\cap \cR^{n-1}_\epsilon(X)$, but they may not equal, because a point in $G_X^2\cap \cR^{n-1}_\epsilon(X)$ may not be glued with any point in $\cR^{n-1}_\epsilon(X)$. We will classify the gluing of $G_Y^2(\epsilon)$ as local separable. See Lemma \ref{l:loc.Pet}.
  \item $F_Y(\epsilon)=\{y\in F_Y\colon f^{-1}(y)\in \cR^{n-1}_\epsilon(X)\}$. Let $F_X(\epsilon)=f^{-1}(F_Y(\epsilon))$. We will classify the gluing near $F_X(\epsilon)\cap \cR^{n-1}_\epsilon(X)$ as involutional. See Lemma \ref{l:inv.glue}.
\end{enumerate}

%


Let $\cU=f(X^\circ)\cup f(F_X^\circ)\cup F_Y(\epsilon)\cup G_Y^2(\epsilon)$ and $\cA=Y\setminus\cU$. It's obvious that $f(X^\circ)\cup f(F_X^\circ)\in\Alex_{loc}^n(\kappa)$ because $f|_{X^\circ}$ is an isometry in terms of the intrinsic metrics. We will first prove $G_Y^2(\epsilon)\in\Alex_{loc}^n(\kappa)$ in Section \ref{s:sep.glue} using Lemma \ref{l:loc.Pet} and Theorem \ref{t:pet.glu}, which also implies Lemma \ref{l:glue.tcone}. The key result is Lemma \ref{l:loc.Pet}, which classifies the gluing over $G_X^2(\epsilon)$ as local separable, so that we can apply Theorem \ref{t:pet.glu}. 
We will prove $F_Y(\epsilon)\in\Alex_{loc}^n(\kappa)$ in Section \ref{s:inv.glue}, which follows from Theorem \ref{t:Self.Glue} and Lemma \ref{l:inv.glue}. In fact, the latter result classifies the gluing near $F_X^2(\epsilon)$ as local involutional, so that we can apply Theorem \ref{t:Self.Glue}. After all these have been done, the result in Step 1-(a) follows from the definition of $\cU$. The results in Step 1-(b) follows from Lemma \ref{l:glue.tcone}. In particular, Theorem \ref{t:loc_glue} follows from Lemma \ref{l:loc.Pet} and Lemma \ref{l:inv.glue}.

For Step 2, we will prove Theorem \ref{t:glob.des} in Section \ref{s:glob}. In Section \ref{s:pf.main}, we combine the results from the previous sections to prove Lemma \ref{l:glue.disc.st}, which implies Theorem \ref{t:glue.disc} and Theorem \ref{t:glue.dim2}.

We would like to point out that the discreteness of $\partial X\setminus\cR^{n-1}_\epsilon(X)$ is only used for Step 2. The results in Section \ref{s:sep.glue} and Section \ref{s:inv.glue} hold for general $X\in\Alex_\amalg^n(\kappa)$ in all dimensions. 

\section{Separable Gluing}\label{s:sep.glue}

This section is dedicated to the proof of $G_Y^2(\epsilon)\in\Alex_{loc}^n(\kappa)$ and Lemma \ref{l:glue.tcone}. To do this, we will need to prove the key Lemma \ref{l:loc.Pet}, which classifies the local gluing structure of $G_X^2(\epsilon)$. As we pointed out in the introduction, points $p$ gluing with $q$ doesn't imply that the neighborhoods of $p$ and $q$ are glued via a separable local isometry. Lemma \ref{l:loc.Pet} says that the gluing of $G_X^2(\epsilon)$ is in fact locally separable. In particular, we have that $G^{2}_{X}(\epsilon)$ is open.

%

We first prove a gluing lemma on the space of directions.

\begin{lemma}\label{l:glue.spd}
  The following hold if $T_{f^{-1}(y)}(X)$ glues to an Alexandrov space $Z$.
  \begin{enumerate}
  \renewcommand{\labelenumi}{(\arabic{enumi})}
  \setlength{\itemsep}{1pt}
    \item Let $x_1,x_2\in f^{-1}(y)$. If $\xi_1\in\Sigma_{x_1}(X)$ and $\xi_2\in \Sigma_{x_2}(X)$ are glued, then there are $T>0$ and geodesics $\gamma_i\colon[0,T]\to T_{x_i}(X)$ with $\gamma_i(0)=x_i^*$ and $\gamma_i'(0)=\xi_i$, $i=1,2$, so that $\gamma_1(t)$ is glued with $\gamma_2(t)$ for every $0\le t\le T$. Moreover, $T$ continuously depends on $\xi_1$.
    \item $\Sigma_{f^{-1}(y)}(X)$ glues to the connected Alexandrov space $\Sigma_\alpha(Z)\in\Alex^{n-1}(1)$, where $\alpha\in Z$ is glued from cone points in $T_{f^{-1}(y)}(X)$. Moreover, the gluing structure on $\Sigma_{f^{-1}(y)}(X)$ is uniquely determined by the gluing of $T_{f^{-1}(y)}(X)$.
    \item Let $a_i, b_i\in X$ with $f(a_i)=f(b_i)$, $i=1,2,\dots$. If $a_i\to a$, then there is a subsequence $\{b_{i_k}\}\subseteq\{b_i\}$ and $b\in X$ so that $\dsp\lim_{k\to\infty}b_{i_k}= b$.
\end{enumerate}
\end{lemma}
\begin{proof}
  Consider the tangent cone $T_\alpha(Z)$. By \cite{NLi15b}, We have that the disjoint union of tangent cones $\dsp\underset{f(x)=y}\amalg T_{x^*}(T_{f^{-1}(y)}(X))$ glues to the tangent cone $T_\alpha(Z)\in\Alex^n(0)$. Note that $T_\alpha(Z)=C(\Sigma_\alpha)$ is a metric cone. We have the property that if $f(x_1)=f(x_2)$ and $u_1\in T_{x_1^*}(T_{f^{-1}(y)}(X))$ is glued with $u_2\in T_{x_2^*}(T_{f^{-1}(y)}(X))$, then the entire rays passing through $u_1$ and $u_2$ are glued by path isometry. Since the gluing on $\dsp\underset{f(x)=y}\amalg T_{x^*}(T_{f^{-1}(y)}(X))$ is a rescaling limit of the gluing on metric cones $T_{f^{-1}(y)}(X)$. Therefore (1) holds.

  Note that $T_\alpha(Z)=C(\Sigma_\alpha(Z))$ is a connected metric cone. Thus $\dsp\underset{f(x)=y}\amalg\Sigma_{x^*}(T_{f^{-1}(y)}(X))$ glues to the connected Alexandrov space $\Sigma_\alpha(Z)$. Moreover, because $T_{x^*}(T_{f^{-1}(y)}(X))$ is simply the rescaling of $T_{x}(X)$ if $f(x)=y$, by (1), we have that $\Sigma_{x^*}(T_{f^{-1}(y)}(X))$ is isometric to $\Sigma_{x}(X)$ and the gluing of $\dsp\underset{f(x)=y}\amalg\Sigma_{x^*}(T_{f^{-1}(y)}(X))$ is the same as the gluing of $\Sigma_{f^{-1}(y)}(X)$.


  To prove (3), we assume contradictively that $\dsp\inf_{i\neq j}\{d(b_i,b_j)\}>0$. Consider the gluing of tangent cone $T_{f^{-1}(f(a))}(X)$ with respect to the rescalling sequence $(X, f^{-1}(f(a)),(d(a,a_i))^{-\frac12}d)$. On the tangent cone $T_a(X)$ with cone point $a^*$, there are sequences $a_i'\to a^*$ and $b_i'\in T_{f^{-1}(f(a))}(X)$ so that each of $a_i'$ is glued with $b_i'$ and $B_1(b_i')\cap B_1(b_j')=\varnothing$ for every $i\neq j$. Note that $d(f_\infty(a^*),f_\infty(b_i'))= d(f_\infty(a^*),f_\infty(a_i'))\le 1$ and thus $B_2(f_\infty(a^*))\supseteq B_1(f_\infty(b_i'))$ for every $i$. We have
  \begin{align}
    C(n)\ge\Vol(B_2(f_\infty(a^*))\ge \Vol\left(\bigcup_{i=1}^\infty B_1(f_\infty(b_i'))\right)
    =\Vol\left(\bigcup_{i=1}^\infty B_1(b_i')\right)
    = \sum_{i=1}^\infty\Vol(B_1(b_i')).
  \end{align}
  This is a contradiction since (\ref{intro.e1}) holds.
\end{proof}

\begin{remark}
  Lemma \ref{l:glue.spd} (3) may not be true for non-compact Alexandrov spaces if we only assume that the gluing is by path isometry and along the boundary. For example, we glue square $[0,1]\times[0,1]$ with strip $[0,1]\times[0,\infty]$ by equivalence relation $(1,\frac1n)\sim(0,n)$.
\end{remark}




%
%

\begin{lemma}\label{l:vol}
Let $X\in \Alex^{n}(k)$ and $\partial X\ne \varnothing$. For any $\epsilon>0$, there exists $\delta(\epsilon,n)>0$ so that the following holds. If $p\in \cR^{n-1}_\delta(X)$, then there is $r>0$ so that for every $x\in B_r(p)\cap\partial X$, we have $x\in \cR^{n-1}_\epsilon(X)$ and
     $$\Vol (\Sigma_{x}(X)) \ge \frac12\Vol(\dS_1^{n-1})-\epsilon.$$
\end{lemma}
\begin{proof}
This easily follows from Bishop-Gromov volume comparison and the volume continuity on Alexandrov spaces.
\end{proof}

\begin{lemma}\label{l:G2-1.dense}
  We have $|f^{-1}(y)|\le N_0(X)<\infty$ for every $y\in Y$ and $\cH^{n-2}(B_1\cap G_X\setminus G_X^2(\epsilon))<C(n,\epsilon)$.
\end{lemma}

\begin{proof}
  The proof is similar to the proof of Theorem 1.1(4) in \cite{NLi15b}. We first prove $|f^{-1}(y)|\le N_0(X)$. By Lemma \ref{l:glue.spd} (2), we have
  \begin{align}
    \Vol(\dS_1^{n-1})
  &\ge\Vol(\Sigma_{y})\ge\sum_{f(x)=y}\Vol(\Sigma_{x}),
  \label{l:Rn-1.dense.e2}
  \end{align}
  which implies the desired result by (\ref{intro.e1}). 

  By Corollary 1.4 in \cite{LiNa19}, we have $\cH^{n-2}(B_1\setminus\cR^{n-1}_\epsilon(X))<C(n,\epsilon)$. Because the gluing is by path isometry, it remains to show that for every $p\in G_X\setminus G_X^2(\epsilon)$ we have $f^{-1}(f(p))\setminus\cR^{n-1}_\epsilon(X)\neq\varnothing$. Note that $G_X\setminus G_X^2(\epsilon)\subseteq G_X^2\setminus G_X^2(\epsilon)\cup(\cup_{m\ge 3}G_X^m)$, we have the following two cases.
  \begin{enumerate}
  \renewcommand{\labelenumi}{(\roman{enumi})}
  \setlength{\itemsep}{1pt}
  \item If $p\in G_X^2\setminus G_X^2(\epsilon)$, then $f^{-1}(f(p))=\{p,q\}$. By the definition of $G_X^2(\epsilon)$, either $p\notin \cR^{n-1}_\epsilon(X)$ or $q\notin \cR^{n-1}_\epsilon(X)$.
  \item If $|f^{-1}(f(p))|\ge 3$, but $f^{-1}(f(p))\subseteq \cR^{n-1}_\epsilon(X)$, by Lemma \ref{l:glue.spd} (2)  and Lemma \ref{l:vol}, we have the following contradictive estimation for $\epsilon>0$ small.
    \begin{align}
    \Vol(\dS_1^{n-1})
  &\ge\Vol(\Sigma_{f(p)})\ge\sum_{q\in f^{-1}(f(p))}\Vol(\Sigma_{q})
  \notag\\
  &\ge \sum_{f(q)=f(p)}\frac25\Vol(\dS_1^{n-1}) \ge \frac65\Vol(\dS_1^{n-1}).
  \label{l:Rn-1.dense.e1}
  \end{align}
\end{enumerate}
\end{proof}

Let $p\in A$. We let $B_{r}^{A}(p)\subseteq A$ denote the metric ball in $A$ centered at $p$ with respect to the intrinsic metric of $A$. In particular, $B_r^{\partial X} (p)$ is the $r$-ball contained in $\partial X$ with respect to the boundary intrinsic metric.

\begin{lemma}\label{l:loc.Pet}
There exists $\epsilon_0(X)>0$ small so that the following holds for any $0<\epsilon<\epsilon_0$. For any $y\in G_Y^2(\epsilon)$ and $f^{-1}(y)=\{p_1,p_2\}\subseteq \cR^{n-1}_\epsilon(X)$, there exists $R>0$ so that $B_{R}(p_1)\cap B_{R}(p_2)=\varnothing$ and the gluing over $B_{R}^{\partial X}(p_1)$ and $B_{R}^{\partial X}(p_2)$ is induced by an intrinsic isometry $\phi\colon B_{R}^{\partial X}(p_1)\to B_{R}^{\partial X}(p_2)$.
\end{lemma}
\begin{proof}
  By Lemma \ref{l:vol}, we first choose $\epsilon, R_1>0$ so that $x\in \cR^{n-1}_\epsilon(X)$ and $\Vol (\Sigma_{x}(X)) \ge \frac12\Vol(\dS_1^{n-1})-\epsilon$
  for every $x\in B_{10R_1}^{\partial X}(p_1)\cup B_{10R_1}^{\partial X}(p_2)$ and $B_{10R_1}^{\partial X}(p_1)\cap B_{10R_1}^{\partial X}(p_2)=\varnothing$.
  We claim that there is $0<R_2<R_1$ so that for every point $x\in B_{10R_2}^{\partial X}(p_1)$, we have $f^{-1}(f(x))\subseteq B_{R_1}^{\partial X}(p_1)\cup B_{R_1}^{\partial X}(p_2)$. If this is not true, then there are glued sequences $x_i\to p_1$ and $y_i\notin B_{R_1}^{\partial X}(p_1)\cup B_{R_1}^{\partial X}(p_2)$ with $f(x_i)=f(y_i)$, $i=1,2,\dots$. By Lemma \ref{l:glue.spd} (3), passing to a subsequence, we have $y_i\to y\notin B_{R_1}^{\partial X}(p_1)\cup B_{R_1}^{\partial X}(p_2)$. Now we have $f^{-1}(f(p_1))=\{p_1, p_2, y\}$. By Lemma \ref{l:glue.spd} and Lemma \ref{l:vol},
  \begin{align*}
    \Vol(\dS_1^{n-1})
  &\ge\Vol(\Sigma_{f(p_1)})\ge\Vol(\Sigma_{p_1})+\Vol(\Sigma_{p_2})+\Vol(\Sigma_{y})
  \\
  &\ge \Vol(\dS_1^{n-1})-2\epsilon+\Vol(\Sigma_{y}),
  \end{align*}
  where $\epsilon\to 0$ as $\delta\to 0$. This leads to a contradiction if $\delta>0$ is chosen small. By a similar volume argument, we have $|f^{-1}(f(x))|\le 2$ for every $x\in B_{10R_2}^{\partial X}(p_1)\cup B_{10R_2}^{\partial X}(p_2)$.

  We prove the desired result by induction on dimension. The statement is obvious for $n=1$. Assume that it is true for $(n-1)$-dimensional spaces. Let $\cC=\{x\in B_{R_2}^{\partial X}(p_1)\colon f^{-1}(f(x))\subseteq B_{R_1}^{\partial X}(p_1)\}$ be the collection of points in $B_{R_2}^{\partial X}(p_1)$ which are not glued with any point in $B_{R_1}^{\partial X}(p_2)$. We shall show that $\cC=\varnothing$. Suppose that this is not true.

  We first show that $\cC$ is open if $\cC\neq\varnothing$. 
  If $x\in\cC$ is not an interior point of $\cC$, then there exists a sequence $x_i\to x$ with $f^{-1}(f(x_i))=\{x_i, x_i'\}$ where $x_i'\notin B_{R_1}^{\partial X}(p_1)$. Passing to a subsequence and let $\dsp x'=\lim_{i\to\infty}x_i'\notin B_{R_1}^{\partial X}(p_1)$. However, we have $\dsp f(x')=\lim_{i\to\infty}f(x_i')=\lim_{i\to\infty}f(x_i)=f(x)$. This contradicts to the assumption $x\in\cC$.

  We now find contradictions using a point-picking technique. Given $r>0$, define function $g_r\colon \partial X\to\dR$ as
  \begin{align}
    g_r(x)&=\renewcommand{\arraystretch}{1.5}
   \left\{\begin{array}{@{}l@{\quad}l@{}}
    \dsp r^{-1}\cdot\sup_{B_{\rho}^{\partial X}(w)\subseteq B_r^{\partial X}(x)\cap \cC}\{\rho\} & \text{ if there is } B_{\rho}^{\partial X}(w)\subseteq B_r^{\partial X}(x)\cap \cC,
    \\
    0 & \text{otherwise}.
  \end{array}\right.
  \end{align}
  It's clear that $g_r$ is a continuous function in $x$. We claim that if $z\notin\cC$, then there exists $r_0>0$ so that $g_r(z)<\frac23$ for every $0<r<r_0$. Given $\delta>0$ small, by \cite{LiNa19}, there exists $r_0>0$ so that $d_{GH}(B_{s}(z), B_{s}(z^*))<\delta s$ for every $0<s<10r_0$, where $z^*\in T_z(X_1)$ is the cone point. Let $0<r<r_0$ and $\rho=g_r(z)>0$. Then there is $w\in B_r^{\partial X}(z)$ so that $B_\rho^{\partial X}(w)\subseteq B_r^{\partial X}(z)\cap\cC$. The intuition is that if $\frac\rho r> \frac12+100\delta$, then $B_\rho^{\partial X}(w)$ will contain $z\notin\cC$.

  By the almost metric cone structure, there exists $y\in B_{\rho}^{\partial X}(w)\subseteq B_r^{\partial X}(z)$ so that
  \begin{align}
    d(w,y)>(1-10\delta)\rho
    \label{l:loc.Pet.e1}
  \end{align}
  and
  \begin{align}
    d(z,w)+d(w,y)<(1+10\delta)d(z,y)\le (1+10\delta)r.
    \label{l:loc.Pet.e2}
  \end{align}
  Note that $z\notin\cC$. Thus $z\notin B_\rho^{\partial X}(w)$ and $d(z,w)\ge\rho$. Combine this with (\ref{l:loc.Pet.e1}) and (\ref{l:loc.Pet.e2}). We get
  \begin{align}
    (2-10\delta)\rho<(1+10\delta)r.
    \label{l:loc.Pet.e3}
  \end{align}
  The claim is proved by choosing $\delta>0$ small.

  Now because $\cC$ is open and $\cC\neq\varnothing$, we can choose $y_1\in \cC$, $z_1\notin\cC$ and $0<r_1<\frac14 R_2$, so that
  \begin{align}
    g_{r_1}(y_1)=1   \text{ and } g_{r_1}(z_1)<\frac23.
  \end{align}
  Because $g_{r_1}$ is continuous, there exists $x_1\in\geod{y_1,z_1}\subseteq B_{R_2}^{\partial X}(p_1)$ with $g_{r_1}(x_1)=\frac56$. That is, there exists $w_1\in B_{r_1}^{\partial X}(x_1)$ so that $B_{\rho}^{\partial X}(w_1)\subseteq B_{r_1}^{\partial X}(x_1)\cap \cC$, where $\rho=\frac56 r_1$. Moreover, there exists $u_1\notin\cC$ so that $d_{\partial X}(u_1,w_1)=\rho$, otherwise it contradicts to the maximum property of $\rho$. Note that $B_{\rho}^{\partial X}(w_1)\subseteq B_{r_1}^{\partial X}(x_1)$. We have $d_{\partial X}(u_1, x_1)\le r_1$. Thus
  \begin{align}
    B_\rho^{\partial X}(w_1)\subseteq B_{r_1}^{\partial X}(x_1)\cap \cC\subseteq B_{2r_1}^{\partial X}(u_1)\cap \cC.
  \end{align}
  and $d_{\partial X}(u_1,p_1)\le r_1+R_2$.
  This implies $g_{2r_1}(u_1)\ge \frac5{12}$. Similarly as before, there exist $y_2\in B_{2r_1}^{\partial X}(u_1)\cap\cC$, $z_2=u_1\notin\cC$ and $0<r_2<\frac1{4} r_1$, so that
  \begin{align}
    g_{r_2}(y_2)=1   \text{ and } g_{r_2}(z_2)<\frac23.
  \end{align}
  Recursively apply the same argument, we obtain sequences $u_i$, $w_i$ and $r_i$, $i=1,2,\dots$ with
  \begin{enumerate}
  \renewcommand{\labelenumi}{(\roman{enumi})}
  \setlength{\itemsep}{1pt}
  \item $u_{i+1}\in B_{2r_i}(u_i)\setminus\cC$ and $r_{i+1}<\frac1{4}r_i$,
  \item $d_{\partial X}(u_i,w_i)=\frac56 r_i$ and $B_{\frac56 r_i}^{\partial X}(w_i)\subseteq B_{2r_i}^{\partial X}(u_i)\cap\cC$.
\end{enumerate}

  Note that $d_{\partial X}(u_i,p_1)\le \sum r_i+R_2\le \frac43R_2$. Passing to a subsequence, we get $\dsp u=\lim_{i\to\infty}u_i\in B_{2R_2}(p_1)$. It's clear that $u\notin\cC$ because $\cC$ is open. Note that $\dsp d_{\partial X}(u,u_i)\le\sum_{j\ge i} 2r_j<6r_i$ for every $i$. We have $B_{2r_i}^{\partial X}(u_i)\subseteq B_{8r_i}^{\partial X}(u)$. Thus for each $i\ge 1$, the ball
  \begin{align}
    B_{\frac56 r_i}^{\partial X}(w_i)\subseteq B_{8r_i}^{\partial X}(u)\cap\cC.
    \label{l:loc.Pet.e4}
  \end{align}

  Now rescale $(X,u)$ by $r_i^{-1}$ and let $i\to\infty$. Passing to subsequences, we have the following Gromov-Hausdorff convergence:
    \begin{enumerate}
  \renewcommand{\labelenumi}{(\roman{enumi})}
  \setlength{\itemsep}{1pt}
    \setcounter{enumi}{2}
  \item $(X, u, r_i^{-1}d)\to (T_u(X),u^*,d_\infty)$, where $u^*\in T_u=T_u(X)$ is the cone point,
  \item $(B_{8r_i}^{\partial X}(u), r_i^{-1}d)\to (B_8^{\partial T_u}(u^*),d_\infty)$,
  \item $(B_{\frac56r_i}^{\partial X}(w_i), r_i^{-1}d)\to (B_{\frac56}^{\partial T_u}(w),d_\infty)$,
  \item $(\cC, u, r_i^{-1}d)\to (\cC_\infty,u^*,d_\infty)$.
\end{enumerate}
  Because $u_i\notin\cC$, that is, every $u_i$ is glued with another point $v_i\in B_{R_1}^{\partial X}(p_2)$, passing to a subsequence, there is $\dsp v=\lim_{i\to\infty}v_i\in B_{R_1}(p_2)$ which is glued with $u\in B_{R_2}(p_1)$. By Lemma \ref{l:glue.spd}, $\Sigma_u(X)$ is glued with $\Sigma_v(X)$ to a connected Alexandrov space.

  We further show that every point in $\partial\Sigma_u(X)$ is glued with some point in $\partial\Sigma_v(X)$, by the inductive hypothesis. Let $\cD\subseteq \partial\Sigma_u(X)$ be the set of points which are glued with some points in $\partial\Sigma_v(X)$. First, $\cD\neq\varnothing$ by the assumption. By the definition, $\cD$ is closed. Now we show that $\cD$ is also open. Let $\xi\in \partial \Sigma_u(X)$ be glued with $\eta\in \partial \Sigma_v(X)$. Note that $\partial \Sigma_u(X)\amalg \partial \Sigma_v(X)= \cR^{n-1}_\delta(\Sigma_u(X)\amalg \Sigma_v(X))$ for some $\delta=\delta(\epsilon,n)>0$ small. We have $\xi,\eta\in \cR^{n-1}_\delta(\Sigma_u(X)\amalg \Sigma_v(X))$. Since $\Sigma_u(X)\amalg \Sigma_v(X)$ glues to an Alexandrov space, by \cite{NLi15b}, we have that the tangent cones at every pair of glued points in $\Sigma_u(X)\amalg \Sigma_v(X)$ are also glued to an Alexandrov space. Now we can apply the inductive hypothesis to get $R'>0$ so that $B_{R'}(\xi)$ is glued with $B_{R'}(\eta)$ by an isometry between $B_{R'}^{\partial \Sigma_u(X)}(\xi)$ and $B_{R'}^{\partial \Sigma_v(X)}(\eta)$.

  By the definition of $\cC$ and (\ref{l:loc.Pet.e4}), we have that the points in $B_{5/6}^{\partial T_u(X)}(w)$ are not glued with any point in $T_{v}(X)$. By the construction, we have $d(u^*,w)=\frac56$ and $B_{5/6}^{\partial T_u}(w)\subseteq B_{8}^{\partial T_u}(u^*)\cap\cC_\infty$. Because $T_u$ is a metric cone, we have that for any $q\in T_u$, if $d(u^*,q)\le\frac5{12}$ and $\ang{u^*}{w}{q}<\frac\pi6$, then $d(w,q)<\frac56$. That is, geodesic $\geodic{u^*,q}\subseteq B_{5/6}^{\partial T_u}(w)\subseteq B_{8}^{\partial T_u}(u^*)\cap\cC_\infty$. This implies that the direction $\uparrow_{u^*}^w\,\in \Sigma_{u}(X)$ is not glued with any direction in the space of directions $\Sigma_{v}(X)$. This is a contradiction.
\end{proof}

By the same argument, we have the following result. It can be viewed as a rigidity theorem on the gluing of Alexandrov spaces which contains no singular point.

\begin{theorem}\label{t:half.sph.full.glue}
  Let $X\in\Alex_\amalg^n(\kappa)$, $Y\in\Alex^n(\kappa')$ and $f\colon X\to Y$ be a 1-Lipschitz volume preserving map. There exists $\epsilon=\epsilon(n)>0$ so that if $\cR^{n-1}_\epsilon(X)=\partial X$, then one of the following holds.
  \begin{enumerate}
  \renewcommand{\labelenumi}{(\arabic{enumi})}
  \setlength{\itemsep}{1pt}
  \item $X$ is connected and $f$ is an isometry.
  \item $X=X_1\amalg X_2$ has only two components and there is an isometry $\phi\colon\partial X_1\to\partial X_2$ so that $f$ is induced by the gluing $x\sim\phi(x)$ and $Y\in\Alex^n(\kappa)$.
\end{enumerate}
\end{theorem}

Combine Lemma \ref{l:G2-1.dense} and Lemma \ref{l:loc.Pet}, we have the following result.

\begin{corollary}\label{l:Rn-1.dense}
  $G^{2}_{X}(\epsilon)$ is open and dense in $G_X$.
\end{corollary}

Combine Theorem \ref{t:pet.glu} and Lemma \ref{l:loc.Pet}, we have the following result.

\begin{corollary}\label{c:sep.glue.Alex}
$G_Y^2(\epsilon)\in\Alex_{loc}^n(\kappa)$.
\end{corollary}

Now we are able to prove Lemma \ref{l:glue.tcone} using Lemma \ref{l:glue.spd} and Lemma \ref{l:loc.Pet},

\begin{proof}[Proof of Lemma \ref{l:glue.tcone}]
  Given $A\subseteq\Sigma_p(X)$, we let $\text{Ray}_{A}\colon[0,\infty)\to T_p(X)$ denote the collection of unit speed geodesics starting from $\text{Ray}_{\xi}(0)=p^*$ and along directions $\xi\in A$. Let $x_1, x_2\in f^{-1}(y)$. We say that the rays $\text{Ray}_{\xi_1}$ and $\text{Ray}_{\xi_2}$ are glued for $0< t< T$ if $\text{Ray}_{\xi_1}(t)$ is glued with $\text{Ray}_{\xi_2}(t)$ for every $0<t< T$ by a path isometry.

  We first show that if $\xi_1\in \cR^{n-2}_\epsilon(\Sigma_{x_1}(X))$ and $\xi_2\in \cR^{n-2}_\epsilon(\Sigma_{x_2}(X))$ are glued, then the rays $\text{Ray}_{\xi_1}$ and $\text{Ray}_{\xi_2}$ are glued for $0<t<\infty$. By Lemma \ref{l:loc.Pet}, there is $r>0$ so that $B_r(\xi_1)$ is glued with $B_r(\xi_2)$ by an intrinsic isometry $\phi\colon B_r^{\partial\Sigma_{x_1}}(\xi_1)\to B_r^{\partial\Sigma_{x_2}}(\xi_2)$. In the tangent cones $T_{f^{-1}(y)}(X)$, by Lemma \ref{l:glue.spd}, there is $T>0$, so that the every ray $\text{Ray}_{\eta_1}$ with $\eta_1\in B_r^{\partial\Sigma_{x_1}}(\xi_1)$ is glued and only glued with a ray $\text{Ray}_{\eta_2}$ with $\eta_2\in B_r^{\partial\Sigma_{x_2}}(\xi_2)$ for $0<t< T$.

  Now we prove by an open-close argument that the above gluing property remains true for $0<t<\infty$. The statement is obviously true for $0<t\le T$. Thus the gluing is induced by an intrinsic isometry $\psi_1\colon \text{Ray}_{B_r^{\partial\Sigma_{x_1}}(\xi_1)}[0,T]\cap\partial T_{x_1}(X)\to \text{Ray}_{B_r^{\partial\Sigma_{x_2}}(\xi_2)}[0,T]\cap\partial T_{x_2}(X)$. Not losing generality, we can assume $B_r^{\partial\Sigma_{x_i}}(\xi_i)\subseteq \cR^{n-2}_\epsilon(\Sigma_{x_i}(X))$, $i=1,2$. Therefore, if $\eta_i\in B_r(\xi_i)$, then $\text{Ray}_{\eta_i}(T)\in \cR^{n-2}_\epsilon(T_{x_i}(X))$ for the same $\epsilon$ because of the cone structure.
  By Lemma \ref{l:loc.Pet} again, there is $r'>0$ so that the gluing of $B_{r'}(\text{Ray}_{\eta_1}(T))$ and $B_{r'}(\text{Ray}_{\eta_2}(T))$ is induced by an isometry $\psi_2\colon B_{r'}^{\partial T_{x_1}}(\text{Ray}_{\eta_1}(T))\to B_{r'}^{\partial T_{x_2}}(\text{Ray}_{\eta_2}(T))$.
  Note that $\psi_2=\psi_1$, restricted on $B_{r'}^{\partial T_{x_1}}(\text{Ray}_{\eta_1}(T))\cap \text{Ray}_{B_r(\xi_1)}[0,T]$ and $\psi_2$ maps geodesics to geodesics. We have that on the gluing of $B_{r'}(\text{Ray}_{\eta_1}(T))$ and $B_{r'}(\text{Ray}_{\eta_2}(T))$, the rays are glued with rays, as a continuation of the gluing between $\text{Ray}_{B_r(\xi_1)}[0,T]$ and $\text{Ray}_{B_r(\xi_2)}[0,T]$. Because the choice of $r'$ continuously depends on $\eta_1$, there is $r_0>0$ so that every ray directed from $B_r^{\partial\Sigma_{x_1}}(\xi_1)$ is glued and only glued with a ray directed from $B_r^{\partial\Sigma_{x_2}}(\xi_2)$ for $0<t< T+r_0$.

  By Corollary \ref{l:Rn-1.dense}, we have that the gluing on $T_{f^{-1}(y)}(X)$ is the completion of the gluing on $G_X^2(\epsilon)$. Therefore, $T_{f^{-1}(y)}(X)$ glues along the rays, which gives a metric cone $C(\Sigma)=Z$ with $\Sigma\in\Alex^{n-1}(1)$ since $C(\Sigma)\in\Alex^n(0)$.
\end{proof}

\section{Involutional Gluing}\label{s:inv.glue}

This section is dedicated to the proof of $F_Y(\epsilon)\in\Alex_{loc}^n(\kappa)$. Again, let us first classify the gluing structure near the points in $F_Y(\epsilon)$.
\begin{lemma}\label{l:inv.glue}
  There exists $\epsilon=\epsilon(n)>0$ small so that for any $p\in F_X(\epsilon)$, there exists $r>0$ and an isometric involution $\psi\colon B_r^{\partial X}(p)\to B_r^{\partial X}(p)$ with $\psi(p)=p$, so that the gluing over $B_r^{\partial X}(p)$ is induced by the identification $x\sim\phi(x)$.
\end{lemma}
\begin{proof}
  Let $\epsilon>0$, to be determined latter. We will show that there exists $0<r<r_0$ so that $|f^{-1}(f(x))|\le 2$ for every $x\in B_r(p)$. Then the isometric involution can be defined as $\phi\colon x\mapsto y$ for any $x,y\in B_r^{\partial X}(p)$ with $f(x)=f(y)$.

  We first choose $r_0>0$ small so that $B_{r_0}^{\partial X}(p)\subseteq\cR^{n-1}_\epsilon(X)$. We claim that there exists $0<r<r_0$ so that $f^{-1}(f(x))\subseteq B_{r_0}(p)$ for every point $x\in B_r(p)$. If this is not true, then there exists $p_i\to p$ and $q_i\notin B_{r_0}(p)$ so that $f(p_i)=f(q_i)$. Passing to a subsequence, we assume $q_i\to q\notin B_{r_0}(p)$. Recall that $f$ is a 1-Lipschitz onto map. Thus $f(p)=f(q)$ and $f^{-1}(f(p))\supseteq\{p,q\}$. This contradicts to $p\in F_Y$.

Now we show that $|f^{-1}(f(x))|\le 2$ for every $x\in B_r(p)$. Suppose $\{x_1, x_2, x_3\}\subseteq f^{-1}(f(x))$. By the definition of $F_X(\epsilon)$, we have $\Vol (\Sigma_{x_i}(X)) \ge \frac12\Vol(\dS_1^{n-1})-\epsilon'$ for each $1\le i\le 3$, where $\epsilon'=\epsilon'(n,\epsilon)\to 0$ as $\epsilon\to 0$. Because $\Sigma_{f^{-1}(f(x))}$ are glued to an Alexandrov space $\Sigma\in\Alex^{n-1}(1)$ along their boundaries, we have
  \begin{align*}
    \Vol(\dS_1^{n-1})
    &\ge\Vol(\Sigma)=\Vol(\Sigma_{f^{-1}(f(x))})
    \\
    &\ge\Vol(\Sigma_{x_1})+\Vol(\Sigma_{x_2})+\Vol(\Sigma_{x_3})
    \\
    &\ge \frac32\Vol(\dS_1^{n-1})-3\epsilon'.
  \end{align*}
  This leads to a contradiction for $\epsilon>0$ small.
\end{proof}

The main goal for now is to prove Theorem \ref{t:Self.Glue}. Let us begin with the definition of length preserving lifting.

\begin{definition}\label{d:lplp}
Let $U$ and $V$ be two length metric spaces. An onto map $f\colon U\to V$ is said to satisfy the \emph{length preserving lifting property}, if for any unit speed geodesic $\gamma: [0,T]\to V$ and $\hat p\in f^{-1}(\gamma(0))$, there exists a continuous curve $\hat \gamma:[0, T]\to U$ so that the following hold.
\begin{enumerate}
  \renewcommand{\labelenumi}{(\arabic{enumi})}
  \setlength{\itemsep}{2pt}
  \item $\hat \gamma(0)=\hat p$.
  \item $f(\hat \gamma(t))=\gamma(t)$ for every $t\in [0, T]$.
  \item $\cL(\hat\gamma([0,t]))=\cL(\gamma([0,t]))$ for every $t\in (0, T]$.
\end{enumerate}
The curve $\hat\gamma$ is called a length preserving lifting of $\gamma$ at $\hat p$.
\end{definition}

For example, let $M$ be a compact Riemannian manifold, $G$ be a group acting on $M$ isometrically and $M/G$ be the quotient space. Then the projection map $f\colon M\to M/G$ satisfies the length preserving lifting property. It is an easy exercise to prove the following proposition.

\begin{proposition}\label{p:lift}
  Let $f\colon U\to V$ be an onto map with the length preserving lifting property. Then the following hold.
  \begin{enumerate}
  \renewcommand{\labelenumi}{(\arabic{enumi})}
  \setlength{\itemsep}{2pt}
  \item $d_U(\hat p,\hat q)\ge d_V(f(\hat p), f(\hat q))$, for any $\hat p,\hat q\in U$.
  \item Fix $\hat p\in U$ and $p=f(\hat p)\in V$. For any $q\in V$, there exists $\hat q\in U$ so that $f(\hat q)=q$ and $d_U(\hat p,\hat q)= d_V(p, q)$.
\end{enumerate}
\end{proposition}

The following can be viewed as a generalization of Corollary 4.6 in \cite{BGP}.

\begin{theorem}\label{t:lift.Alex}
Let $U$ and $V$ be length metric spaces and $f\colon U\to V$ be an onto map. If $U$ is an Alexandrov $\kappa$-domain and $f$ satisfies the length preserving lifting property, then $V$ is also an Alexandrov $\kappa$-domain. 
\end{theorem}
\begin{proof}
Given $p\in V$, let $\hat p\in U$ so that $f(\hat p)=p$. By Proposition \ref{p:lift} (2), for any three points $q_{i}\in V$, $i=1,2,3$, there exist lifted points $\hat q_{i}\in U$ with $f(\hat q_{i})=q_{i}$ and $d_U(\hat p,\hat q_i)=d_V(p, q_i)$. By Proposition \ref{p:lift} (1), we have $d_V(q_i, q_j)\le d_U(\hat q_i, \hat q_j)$. Thus $\cang{\kappa}{p}{q_i}{q_j}\le \cang{\kappa}{\hat p}{\hat q_i}{\hat q_j}$. Now we have
\begin{align}
  \sum_{1\le i< j\le 3}\cang{\kappa}{p}{q_i}{q_j}\le \sum_{1\le i< j\le 3} \cang{\kappa}{\hat p}{\hat q_i}{\hat q_j}\le 2\pi.
\end{align}
The last inequality follows from that $U$ is an Alexandrov $\kappa$-domain.
\end{proof}

\begin{proof}[Proof of Theorem \ref{t:Self.Glue}]
Let $U_1=U_2=B_{R/10}(x)$ be two copies of $B_{R/10}(x)$ and identify $\phi$ as an isometry between $U_1\cap\partial X$ and $U_2\cap\partial X$. Let $\widehat U=U_1\amalg U_2/(x\sim\phi(x))$ be the gluing space. By Theorem \ref{t:pet.glu}, we have that $\widehat U$ is an Alexandrov $\kappa$-domain.

Let $\pi\colon U_1\amalg U_2\to\widehat U$ be the projection map induced by the gluing $\phi\colon U_1\cap\partial X \to U_2\cap\partial X$. See Figure \ref{fig:involution}. Let $V=f(B_{R/10}(x))$ and $g\colon U_1\amalg U_2\to V$ be defined as $g|_{U_i}=f_i\colon U_i\to V$, $i=1,2$, where $f_i\equiv f|_{B_{R/10}(x)}$. We have the following properties for map $\hat f=g\circ\pi^{-1}\colon \widehat U\to V$.
  \begin{enumerate}
  \renewcommand{\labelenumi}{(\roman{enumi})}
  \setlength{\itemsep}{2pt}
  \item $\hat f$ is well-defined. This is because if $\pi^{-1}(\hat p)=\{x_1, x_2\}$, then $\phi(x_1)=x_2$ and thus $f(x_1)=f(x_2)$.
  \item $\hat f$ is a length preserving onto. This follows from the definition of $\hat f$ and both $\pi$ and $g$ are length preserving onto.
\end{enumerate}

\begin{figure}
\begin{tikzpicture}[scale=1]
\tikzstyle{every node}=[font=\Small] 
\draw [fill=SkyBlue!90](2, -3)--(-2, -3)arc(180:360:2);
\draw (2, 0)--(-2, 0)arc(180:0:2);
\draw (-1,0.8)..controls(-0.9,0.5)..(-0.8,0);
\draw (-0.8, -3)..controls(-0.6,-4)..(-0.2,-3);
\draw (-0.2, 0)..controls(0.2,1.2)..(1,0);
\draw (1,-3)..controls(1.15,-3.5)..(1.2, -4);
\filldraw[black](-1, 0.8)circle (1pt)[left]node{$p_{1}$};
\filldraw[black](-0.8,0)circle (1pt)[below,xshift=-8]node{$\sigma_0(t_1)$};
\filldraw[black](-0.8,-3)circle (1pt)[above,xshift=-8]node{$\sigma_1(t_1)$};
\filldraw[black](-0.2,-3)circle (1pt)[above,xshift=5]node{$\sigma_1(t_2)$};
\filldraw[black](-0.2,0)circle (1pt)[below,xshift=5]node{$\sigma_2(t_2)$};
\filldraw[black](1,0)circle (1pt)[below,xshift=5]node{$\sigma_2(t_3)$};
\filldraw[black](1,-3)circle (1pt)[above,xshift=5]node{$\sigma_3(t_3)$};

\draw [<->] (0, -1.9) -- (0,-1.1);
\node at (0.2, -1.5)[right]{$\phi$};
\node at(2.5,-2)[above]{$U_{1}\amalg U_{2}$};
\node at(-3,1){$U_{1}$};
\node at(-3,-4){$U_{2}$};

\draw [fill=SkyBlue!90](10, 0)--(6, 0)arc(180:360:2);
\draw (6,0)--(10,0)arc(0:180:2);
\draw [blue](10, 0)--(6, 0);
\draw (7, 0.8)..controls(7.1,0.5)..(7.2, 0)..controls(7.4, -1)..(7.8, 0)..controls(8.2, 1.2)..(9,0)..controls(9.15,-0.5)..(9.2,-1); 
\filldraw[black](7, 0.8)circle (1pt)[left]node{$\hat p$};
\filldraw[black](7.2, 0)circle (1pt)[below,xshift=-10]node{$\hat{\gamma}(t_{1})$};
\filldraw[black](7.8, 0)circle (1pt)[below,xshift=9]node{$\hat{\gamma}(t_{2})$};
\filldraw[black](9, 0)circle (1pt)[above,xshift=8]node{$\hat{\gamma}(t_{3})$};
\draw [->] (3.5, 0) -- (5,0);
\node at(4.2,0.2)[above]{$\pi$};
\node at(10.2, 0)[right, yshift=1]{$\widehat U$};
\draw [->] (8, -3) -- (8,-5); 
\node at(8.6, -4){$\hat{f}$};
\draw [->] (3.4, -2.4) -- (7,-5);
\node at(4.6, -3.8)[below]{$g$};
\draw (8, -8)--(8-1.51, -6.01);
\draw (8, -8)--(8+1.51, -6.01);
\draw[blue] (8,-8)--(8-0.5, -6.328);
\fill[left color=RoyalBlue!70, right color=SkyBlue!90, shading=axis] (8, -6) circle(1.5 and 0.35);
\draw(8, -6) circle(1.5 and 0.35);
\node at(10.2, -7){$V$};

\draw (8.2, -6.5)..controls(8,-6.5)..(7.6, -6.7)..controls(7.2,-7)..(7.85,-7.45)..controls(8,-7.5) and (8.3, -7.55)..(8.3,-7.6);
\filldraw[black](8.2, -6.5)circle (1pt)[below, xshift=-3]node{$p$};
\draw[dotted] (8.3,-7.6)..controls(8.35,-7.7) and (7.65,-7.7)..(7.62,-7.5);
\draw (7.62,-7.5)..controls (7.7, -7) and (8,-6.9)..(8.7,-6.5);
\end{tikzpicture}

\caption{Involution}\label{fig:involution}
\end{figure}
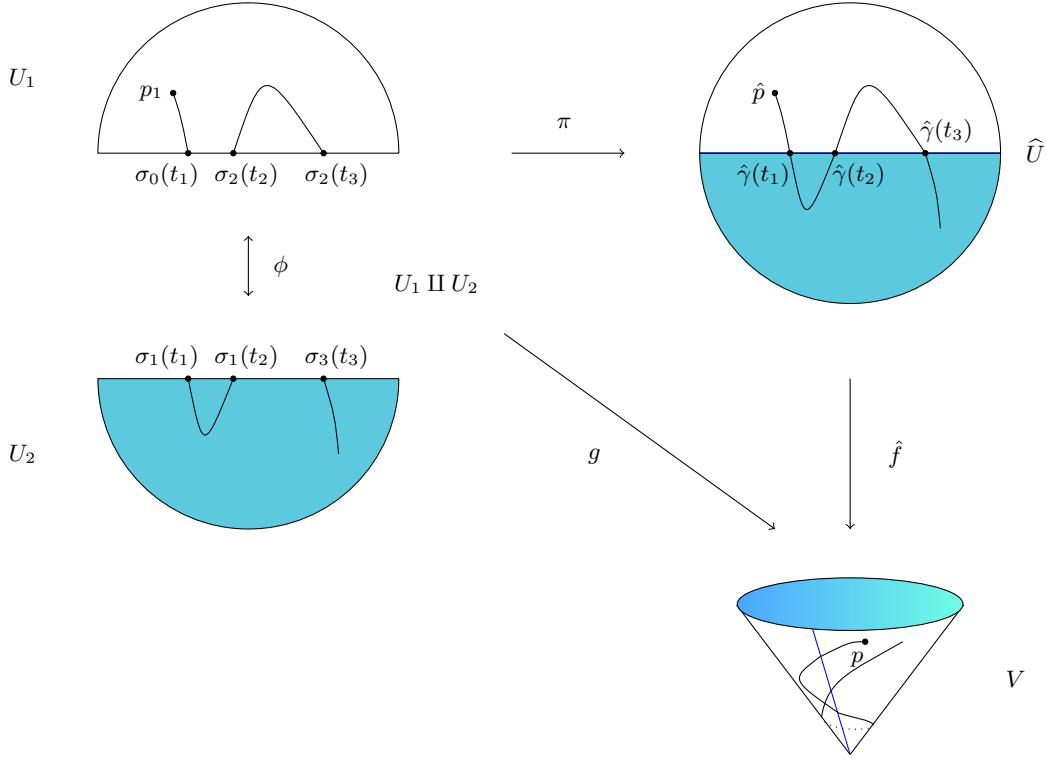
Now we show that $\hat f$ satisfies the length preserving lifting property. Then the desired result follows from Theorem \ref{t:lift.Alex}. Let us first show that the length preserving lifting property holds for any curve $\gamma\colon[0,T]\to V$ with a partition $0=t_0<t_1<t_2<\dots<t_N=T$ for which $\gamma|_{(t_j, t_{j+1})}\subseteq f(U^\circ)$ and $\gamma(t_j)\in f(\partial U)$ for every $0\le j\le N-1$. Let $\hat p\in \hat f^{-1}(\gamma(0))$. We define a sequence of lifting $\sigma_j\colon[t_j,t_{j+1}]\to U_1\amalg U_2$ as follows. Not losing generality, let $p_1\in U_1$ so that $\pi(p_1)= \hat p$.
\begin{enumerate}
  \renewcommand{\labelenumi}{(\roman{enumi})}
  \setlength{\itemsep}{2pt}
  \item Let $\sigma_0(t)=f_1^{-1}(\gamma(t))$ for $t\in (t_0,t_1)$
  \item Suppose that $\sigma_{j-1}(t)=f_i^{-1}(\gamma(t))$ for $t\in (t_{j-1},t_j)$ has been defined, where $i\in\{1,2\}$. Let $\sigma_j(t)=f_{i'}^{-1}(\gamma(t))$ for $t\in (t_j,t_{j+1})$, where $i'\in\{1,2\}\setminus\{i\}$.
  \item Let $\dsp\sigma_j(t_j)=\lim_{t\to t_j^+}\sigma(t)$ and $\dsp\sigma_j(t_{j+1})=\lim_{t\to t_{j+1}^-}\sigma(t)$. By (ii) and $g(\sigma_{j-1}(t_j))=g(\sigma_{j}(t_j))=\gamma(t_j)\in f(\partial U)$, we have $\phi(\sigma_{j-1}(t_j))=\sigma_{j}(t_j)$.
\end{enumerate}
Now $\hat\gamma\colon[0,T]\to \widehat U$, defined by $\hat\gamma|_{[t_j, t_{j+1}]}=\pi\circ\sigma_j$ is a continuous curve. This is because if $\gamma(t_j)\in f(\partial U)$, then $g^{-1}(\gamma(t_j))=\{\sigma_{j-1}(t_{j}), \sigma_{j}(t_{j})\}$, where $\phi(\sigma_{j-1}(t_j))=\sigma_{j}(t_j)$ and $\pi(\sigma_{j-1}(t_j))=\pi(\sigma_{j}(t_j))$.

We claim that $\hat\gamma$ is the desired lifting of $\gamma$ at $\hat p$. Properties (1) and (2) in Definition \ref{d:lplp} are obvious by the definition of $\hat\gamma$. Property (3) follows from that every $\sigma_j$ is a length preserving lifting and $\pi$ is length preserving.

It remains to show that the length preserving lifting property holds for any unit speed geodesic $\gamma\colon[0,T]\to V$. By Definition \ref{d:gluing}, $\gamma$ can be approximated by a sequence of curves $\gamma_k$ that satisfies the following.
\begin{enumerate}
  \renewcommand{\labelenumi}{(\roman{enumi})}
  \setlength{\itemsep}{2pt}
  \item $\dsp\lim_{k\to\infty}\gamma_k=\gamma$.
  \item $\dsp\lim_{k\to\infty}\cL(\gamma_k)=\cL(\gamma)$.
  \item Each of $\gamma_k$ is the image of a finite union of geodesics in $X$.
\end{enumerate}
Note that each of $\gamma_k$ has a length preserving lifting $\hat\gamma_k$ at $\hat p\in \hat f^{-1}(\gamma(0))$. Thus we have $\cL(\gamma_k)=\cL(\hat\gamma_k)$. Passing to a subsequence, let $\dsp\hat\gamma=\lim_{k\to\infty}\hat\gamma_k$. It's clear that $\hat f(\hat\gamma)=\gamma$ and thus $\cL(\hat\gamma)= \cL(\gamma)$, because $\hat f$ is length preserving. Then we have
\begin{align}
  \cL(\gamma)
  &=\lim_{k\to\infty}\cL(\gamma_k)
  =\lim_{k\to\infty}\cL(\hat\gamma_k)
  \ge \cL(\hat\gamma)
  = \cL(\gamma).
\end{align}
Therefore, $\cL(\hat\gamma)=\cL(\gamma)$ and $\hat\gamma$ is a length preserving lifting of $\gamma$ at $\hat p$.
\end{proof}

By Lemma \ref{l:inv.glue} and Theorem \ref{t:Self.Glue}, we have the following main result for this section.

\begin{corollary}\label{c:inv.glue.Alex}
 $F_Y(\epsilon)\in\Alex_{loc}^n(\kappa)$.
\end{corollary}

\section{Globalization over Discrete Singular Points}\label{s:glob}

This section is dedicated to the proof of the following theorem, which implies Theorem \ref{t:glob.des}, due to Remark \ref{r:alex.convex}.

\begin{theorem}\label{t:glob.des.convex}
Let $Y$ be a length metric space with $\cU\subseteq Y$ and $\cA=Y\setminus \cU$. Then $Y\in\Alex^n(\kappa)$ if the following are satisfied.
\begin{enumerate}
\renewcommand{\labelenumi}{(\arabic{enumi})}
\setlength{\itemsep}{1pt}
\item $\cU\in\Alex_{loc}^n(\kappa)$.
\item $\cA$ is discrete and every point $p\in \cA$ is convex.
\end{enumerate}
\end{theorem}

By the following Globalization Theorem, we only need to show that $\cA\in\Alex^n_{loc}(\kappa)$. Note that Theorem \ref{t:glob.des.convex} is also a generalization of Theorem \ref{t:glob.BGP}.

\begin{theorem}[Globalization \cite{BGP}]\label{t:glob.BGP}
  If $Y\in\Alex_{loc}^n(\kappa)$ and $Y$ is complete, then $Y\in\Alex^n(\kappa)$.
\end{theorem}

Let us begin with an important tool in the business of globalization.

\begin{lemma}[Alexandrov's Lemma \cite{AKP}]\label{Alex.lem}
  Let $p,q,s$ and $x\in\geodii{qs}$ be points in $Y$. Then $\cang{\kappa}{q}{p}{x}\ge\cang{\kappa}{q}{p}{s}$ if and only if $\cang{\kappa}{x}{p}{q}+\cang{\kappa}{x}{p}{s}\le\pi$.
\end{lemma}

The following lemma was implicitly proved in \cite{Pet2016} and \cite{NLi15a}. We will state it here and give a proof just for the completeness. Let $\geod{ab}$ denote a geodesic connecting points $a$ and $b$. Let $\geodii{ab}$ denote the interior of $\geod{ab}$. Finally, let $\geodci{ab}=\{a\}\,\cup\,\geodii{ab}$ and $\geodic{ab}=\geodii{ab}\,\cup\,\{b\}$.

\begin{lemma}\label{l:thin.glob}
Let $a,b,c\in Y$, $\kappa\in \mathbb R$. Suppose that there are geodesics $\geod{ab}$ and $\geod{ac}$ so that for every $u\in\geodii{ab}$ and $v\in \geodii{ac}$, there is a geodesic $\geod{uv}$ for which every point on $\geod{uv}$ admits a $\kappa$-domain. Then for any $b_2\in\geodic{ab}$, $b_1\in\geodic{ab_2}$, $c_2\in\geodic{ac}$ and $c_1\in\geodic{ac_2}$, we have
  \begin{align}
    \cang{\kappa}{a}{c_2}{b_2}\le\cang{\kappa}{a}{c_1}{b_1}\le \ang{a}{c}{b}.
  \label{l:thin.glob.e0}
  \end{align}
\end{lemma}

\begin{proof}
  Fix the given $c_1\in\geod{ac_2}$. By the argument in \cite{Pet2016} or \cite{NLi15a}, for any $x\in\geodii{ab}$, we have
  \begin{align}
   \cang{\kappa}{x}{c_1}{y}\le \ang{x}{c_1}{y}
   \label{pf.thin.glob.e1}
 \end{align}
for every $y\in\geodii{ab}$ being close to $x$.
Fix $p\in\geodii{ab_2}$ and $q\in\geodii{pb_2}$ so that $b_1\in\geodii{p, q}$. Let
$$\{p=x_0, x_1, x_2,\dots, x_j=b_1,\dots,x_N=q\}$$
be a partition of $\geod{pq}$ for which (\ref{pf.thin.glob.e1}) holds for every $\{x,y\}=\{x_i, x_{i+1}\}$. For every $0\le i<j\le N$, we will prove
\begin{align}
   \cang{\kappa}{x_i}{c_1}{x_j}\le\dots\le\cang{\kappa}{x_i}{c_1}{x_{i+2}} \le\cang{\kappa}{x_i}{c_1}{x_{i+1}}\le\ang{x_i}{c_1}{x_j}
   \label{pf.thin.glob.e2}
\end{align}
and
\begin{align}
   \cang{\kappa}{x_j}{c_1}{x_i}\le\dots\le\cang{\kappa}{x_j}{c_1}{x_{j-2}} \le\cang{\kappa}{x_j}{c_1}{x_{j-1}}\le\ang{x_j}{c_1}{x_i}
   \label{pf.thin.glob.e3}
\end{align}
by induction. 

First, (\ref{pf.thin.glob.e2}) and (\ref{pf.thin.glob.e3}) hold for $j-i=1$ by the definition of the partition. Now assume that (\ref{pf.thin.glob.e2}) and (\ref{pf.thin.glob.e3}) hold for every $0\le i<j\le N$ with $j-i=1,2,\dots,k$. We now prove them for every $0\le i<j\le N$ with $j-i=k+1$.

By the inductive hypothesis, we have
\begin{align}
   \cang{\kappa}{x_j}{c_1}{x_{j+1}}\le\ang{x_j}{c_1}{x_{j+1}}
   \text{ and }
   \cang{\kappa}{x_j}{c_1}{x_i}\le\ang{x_j}{c_1}{x_i},
   \label{pf.thin.glob.e5}
\end{align}
where the first inequality follows from (\ref{pf.thin.glob.e2}) and the second inequality follows from (\ref{pf.thin.glob.e3}). Because $\{x_j\}\in\Alex_{loc}^n(\kappa)$, we get
\begin{align}
   \cang{\kappa}{x_j}{c_1}{x_{j+1}}+\cang{\kappa}{x_j}{c_1}{x_i}
   \le\ang{x_j}{c_1}{x_{j+1}}+\ang{x_j}{c_1}{x_i}=\pi.
   \label{pf.thin.glob.e6}
\end{align}
Apply Alexandrov's Lemma. We obtain
\begin{align}
   \cang{\kappa}{x_i}{c_1}{x_{j+1}}\le\cang{\kappa}{x_i}{c_1}{x_j}.
   \label{pf.thin.glob.e7}
\end{align}
Combine this with the inductive hypothesis, we prove (\ref{pf.thin.glob.e2}) for $j-i=k+1$. The proof of (\ref{pf.thin.glob.e3}) is similar.

It follows from (\ref{pf.thin.glob.e2}) that
 \begin{align}
   \cang{\kappa}{p}{c_1}{q}\le\cang{\kappa}{p}{c_1}{b_1}
   \label{pf.thin.glob.e8}.
 \end{align}
Let $p\to a$ and $q\to b_2$. We get
\begin{align}
   \cang{\kappa}{a}{c_1}{b_2}\le\cang{\kappa}{a}{c_1}{b_1}
   \label{pf.thin.glob.e10}.
\end{align}
Similarly, we can also prove
\begin{align}
   \cang{\kappa}{a}{c_2}{b_2}\le\cang{\kappa}{a}{c_1}{b_2}
   \label{pf.thin.glob.e11}.
\end{align}
Combine (\ref{pf.thin.glob.e10}) and (\ref{pf.thin.glob.e11}), we get the desired result.
\end{proof}

\begin{proof}[Prove of Theorem \ref{t:glob.des.convex}]
 By Theorem \ref{t:glob.BGP}, we only need to prove $\cA\in\Alex_{loc}^n(\kappa)$. Let $\omega\in\cA$ and $B_r(\omega)$ be an open ball so that $B_r(\omega)\cap \cA=\{\omega\}$. Take $U=B_{r/100}(\omega)$. We will show that $U$ is a $\kappa$-domain.

Note that for any geodesics $\geod{ab}$ and $\geod{ac}$, the angle $\measuredangle\geod{ab}\geod{ac}$ is well-defined as in Definition \ref{def:angle}. If no confusion arises, we will denote $\measuredangle\geod{ab}\geod{ac}$ by $\ang{a}{b}{c}$.

The key is to prove
 \begin{align}
   \cang{\kappa}{a}{b}{c}\le \ang{a}{b}{c}\label{pf.t:glob.des.e1}
 \end{align}
and furthermore
 \begin{align}
   \cang{\kappa}{a}{b}{c}\le \cang{\kappa}{a}{c}{q}\label{pf.t:glob.des.e1_2}
 \end{align}
 for every $q\in\geodii{ab}$.

First, we observe that geodesic $\geod{xy}\subseteq B_r(p)$ for any $x,y\in U$. Therefore $\geod{xy}\cap \cA\subseteq\{\omega\}$. Not losing generality, we assume that $a,b,c$ do not lie on the same geodesic. We prove (\ref{pf.t:glob.des.e1}) by the following three steps.

(Step 1.) We prove (\ref{pf.t:glob.des.e1}) and (\ref{pf.t:glob.des.e1_2}) for $\omega\in\{a,b,c\}$. We claim that for every $u\in\geodii{ab}$ and $v\in \geodii{ac}$, there is no geodesic $\geod{uv}\ni \omega$. If this is true, then the assumption of Lemma \ref{l:thin.glob} is satisfied and (\ref{pf.t:glob.des.e1}) follows.

Now we prove the claim. Suppose it is not true for some  geodesic $\geod{uv}$. If $\omega=a$, because geodesics do not bifurcate at $u$ or $v$, we have that $\geod{ua}\cup\geod{av}=\geod{uv}\subseteq \geod{ba}\cup\geod{ac}$. Therefore, $a$ lies on the geodesic connecting $c$ and $b$. This contradicts to the assumption. The proof for the cases $\omega=b$ or $\omega=c$ is similar.

(Step 2.) We prove (\ref{pf.t:glob.des.e1}) for $\omega\in\geodii{ab}$. By the result of Step 1, we have that
 \begin{align}
   \cang{\kappa}{\omega}{c}{b}\le \ang{\omega}{c}{b}
   \text{ and }
   \cang{\kappa}{\omega}{c}{a}\le \ang{\omega}{c}{a}.
   \label{pf.t:glob.des.e10}
 \end{align}
Because $\omega$ is convex, we have
\begin{align}
  \cang{\kappa}{\omega}{c}{b}+\cang{\kappa}{\omega}{c}{a}\le \ang{\omega}{c}{b}+\ang{\omega}{c}{a}\le\pi.
  \label{pf.t:glob.des.e11}
\end{align}
By Alexandrov's Lemma,
we get
\begin{align}
  \cang{\kappa}{a}{c}{b}\le \cang{\kappa}{a}{c}{\omega}\le \ang{a}{c}{\omega}=\ang{a}{c}{b},
  \label{pf.t:glob.des.e12}
\end{align}
where the second inequality follows from the result of Step 1. Estimate (\ref{pf.t:glob.des.e1_2}) follows from the above inequality and Lemma \ref{l:thin.glob}.

(Step 3.) We prove (\ref{pf.t:glob.des.e1_2}) for the case that $\omega$ is not on any side of geodesic triangle $\triangle abc$. If By Lemma \ref{l:thin.glob}, it suffices to prove (\ref{pf.t:glob.des.e1_2}) for the case that $\omega\in\geodii{cq}$ for some $q\in\geodii{ab}$. By the result of Step 2, we get
\begin{align}
  \cang{\kappa}{q}{c}{a}\le \ang{q}{c}{a}
  \text{ and }
  \cang{\kappa}{q}{c}{b}\le \ang{q}{c}{b}.
  \label{pf.t:glob.des.e13}
\end{align}
Note that $\{q\}\in\Alex_{loc}^n(\kappa)$. Thus
\begin{align}
  \cang{\kappa}{q}{c}{a}+\cang{\kappa}{q}{c}{b}
  \le \ang{q}{c}{a}+\ang{q}{c}{b}=\pi.
  \label{pf.t:glob.des.e15}
\end{align}
By Alexandrov's Lemma, we conclude that
\begin{align}
  \cang{\kappa}{a}{c}{b}\le\cang{\kappa}{a}{c}{q}.
  \label{pf.t:glob.des.e16}
\end{align}
\end{proof}

\begin{remark}\label{r:glob.des.convex}
  Our proof for Theorem \ref{t:glob.des.convex} will have technique difficulties if $\cA$ has higher dimension. For example, the argument in Step 1 wouldn't work and we don't have (\ref{pf.t:glob.des.e10}) to begin with.
\end{remark}

\section{Proof of the Main Results}\label{s:pf.main}

As a summary of Corollary
\ref{c:sep.glue.Alex} and Corollary
\ref{c:inv.glue.Alex}, we have

\begin{lemma}\label{l:decomp.X.U}
$\cU\in\Alex_{loc}^n(\kappa)$.
\end{lemma}

Now we can prove the following lemma, using the way described in Section \ref{s:outline}.

\begin{lemma}\label{l:glue.disc.st}
Let $Y$ be a connected length metric space glued from compact $X\in\Alex_\amalg^n(\kappa)$ along their boundaries and by path isometry. Suppose that for every $y\in Y$ and every limit gluing map $f_\infty$, the tangent cones $T_{f^{-1}(y)}(X)$ glue to an Alexandrov space $Z\in\Alex^n(\kappa)$. If $\partial X\setminus(F_X^\circ\cup F_X(\epsilon)\cup G_X^2(\epsilon))$ is discrete for any $\epsilon>0$, then $Y\in\Alex^n(\kappa)$.
\end{lemma}

\begin{proof}
  Let $\cU=f(X^\circ)\cup F_Y(\epsilon)\cup G_Y^2(\epsilon)$ and $\cA=Y\setminus\cU=f(\partial X\setminus(F_X(\epsilon)\cup G_X^2(\epsilon)))$. By Lemma \ref{l:decomp.X.U}, we have that $\cU\in\Alex_{loc}^n(\kappa)$. By Lemma \ref{l:glue.tcone}, we have that every point in $\cA$ is convex. Note that $\cA$ is assumed to be discrete. Then the desired result follows from Theorem \ref{t:glob.des.convex}.
\end{proof}

\begin{proof}[Proof of Theorem \ref{t:glue.disc}]
  By Lemma \ref{l:glue.disc.st}, it suffices to show that if $H_1=\partial X\setminus(F_X^\circ\cup\cR^{n-1}_\epsilon(X))$ is discrete, then so does $H_2=\partial X\setminus(F_X^\circ\cup F_X(\epsilon)\cup G_X^2(\epsilon))$. Let $H=H_2\setminus H_1$. By the definition of $F_X(\epsilon)$ and $G_X^2(\epsilon)$ and the volume estimate (\ref{l:Rn-1.dense.e1}), we have that for every $x\notin F_X(\epsilon)\cup G_X^2(\epsilon)$, there is $x'\in f^{-1}(f(x))$ so that $x'\notin\cR^{n-1}_\epsilon(X)$. Thus $|H_2|\le m|H_1|$, where $\dsp m=\max_{y\in f(\partial X)}\{|f^{-1}(y)|\}\le C(X)<\infty$.
\end{proof}

\begin{proof}[Proof of Theorem \ref{t:glue.dim2}]
  Note that $G_X\setminus\cR^{1}_\epsilon(X)\subseteq\cS^0_\epsilon(X)$ is discrete for $n=2$. It remains to show that for every $y\in Y$ and every limit gluing map $f_\infty$, the tangent cone $T_{f^{-1}(y)}(X)$ glues to an Alexandrov space $Z\in\Alex^2(\kappa)$. Then the result follows from Theorem \ref{t:glue.disc}.

  We first show that the gluing is continuous in some sense. Let $\epsilon>0$ be determined latter. Given $p\in \partial X$, because $T_p(X)$ is a metric cone, there is $r>0$ small so that $B_r^{\partial X}(p)\setminus\{p\}\subseteq\cR^1_\epsilon(X)$, and moreover for every $0<s<r$, we have
\begin{equation}\label{eq:glue.dim2.1}
d_{GH}(B_{s}(p), B_{s}(p^{*})<s\epsilon
\end{equation}
where $p^*\in T_{p}(X)$ is the cone point. See \cite{LiNa19} for the existence of such an $r$. Note that $\partial\Sigma_p=\{\xi,\eta\}$ is a set of two points. Consider the gradient exponential map $g\exp_{p}(t\xi)$, $0<t\le r$. 
Let $A_\xi(r)=\{g\exp_p(t\xi)\colon t\in (0,r]\}$ be the image of $g\exp_p$ and $A_\eta(r)$ be defined similarly. Then $A_\xi(r)\cup A_\eta(r)=B_r^{\partial X}(p)\setminus\{p\}\subseteq \cR^1_\epsilon(X)$. Let $\cC_\xi=\{x\in A_\xi(r)\colon f^{-1}(f(x))=x\}$ be the set of point which are not glued with any other point in $A_\xi(r)$. We claim that either $\cC_\xi=\varnothing$ or $\cC_\xi= A_\xi(r)$. That is, the gluing is along the gradient curves.

It's clear that $\cC_\xi$ is open in $\partial X$, because $\partial X\setminus\cC_\xi$ is closed. If $\cC_\xi\neq\varnothing$, then it is a union of countable continuous curves with open ends. Let $L=\{g\exp_p(t\xi)\colon t\in (a,b), a>0\}$ be one of such curves. If $\cC_\xi\neq A_\xi(r)$, then neither $g\exp_p(a\xi)$ nor $g\exp_p(b\xi)$ belongs to $\cC_\xi$. Let $x\neq g\exp_p(a\xi)$ be glued with $g\exp_p(a\xi)$.

By the construction, $f(g\exp_p(a\xi))$ is a boundary point. However, we have \begin{align}
    \pi\ge \Theta_{g\exp_p(a\xi)}+\Theta_{x}\ge \pi-\delta(\epsilon)+\Theta_{x},
  \end{align}
  where $\delta(\epsilon)\to 0$ as $\epsilon\to 0$. This leads to a contradiction because  $\inf_{x\in X}\Theta_{x}\ge c(X)>0$. Thus we have either $\cC_\xi=\varnothing$ or $\cC_\xi= A_\xi(r)$.

By the same argument, we have either $\cC_\eta=\varnothing$ or $\cC_\eta= A_\eta(r)$. Note that there is no isolated gluing point in $X$. Therefore $\cC_\xi= A_\xi(r)$ and $\cC_\eta= A_\eta(r)$ can't both be true. Note that Lemma \ref{l:G2-1.dense}, Lemma \ref{l:loc.Pet} and Lemma \ref{l:inv.glue} hold with the assumption $\dsp\sum_{f(x)=y}\Theta_{x}\le2\pi$. Thus in any case, the gradient curves are glued with gradient curves by path isometry if $r$ is chosen small enough. Therefore, the tangent cone $T_{f^{-1}(f(p))}$ is glued to a metric cone $C(\Sigma)$, where the perimeter of $\Sigma$ is equal to $\sum_{f(x)=f(p)}\Theta_{x}$. By the assumptions on the glued cone angles in Theorem \ref{t:glue.dim2}, we have that $\Sigma\in\Alex^1(1)$. Thus $C(\Sigma)\in\Alex^2(0)$
\end{proof}

\end{document}